\documentclass{elsarticle}
\usepackage{fancyhdr}

\usepackage{amsmath}
\usepackage{mathtools}
\usepackage{amsfonts}
\usepackage{amsthm}
\usepackage{amssymb}
\usepackage{color}
\usepackage{dsfont}
\usepackage{geometry}
\newgeometry{tmargin=2.8cm, bmargin=3.5cm, lmargin=2.3cm, rmargin=2.3cm}

\usepackage{color}

\newtheorem{theo}{\bf Theorem}[section]
\newtheorem{coro}{\bf Corollary}[section]
\newtheorem{lem}{\bf Lemma}[section]
\newtheorem{rem}{\bf Remark}[section]
\newtheorem{defi}{\bf Definition}[section]
\newtheorem{ex}{\bf Example}[section]

\newtheorem{prop}{\bf Proposition}[section]
\newtheorem*{oq}{\bf Open question} 

\def\R{{\mathbb{R}}}

\def\r{{\mathbb{R}}}
\def\N{{\mathbb{N}}}
\def\rn{{\mathbb{R}^{N}}}
\def\rnt{{\mathbb{R}^{N+1}}}

\def\VTM{{V_T^M(\Omega)}}
\def\VTMi{{V_T^{M,\infty}(\Omega)}}
\newcommand{\sg}{{\mathrm{sgn}_0^+}}
\newcommand{\supp}{{\mathrm{supp}}}
\newcommand{\essinf}{{\mathrm{ess\,inf}}}
\newcommand{\va} {\vec{a}}

\newcommand{\tm} {\mu}

\newcommand{\gkbm} {(T_k(u))^\bullet_\mu}
\newcommand{\gkm} {(T_k(u))_\mu}

\def\rp{{[0,\infty )}}

\def\ve{{\varepsilon}}
\def\vr{{\varrho}}
\def\vt{{\vartheta^{\tau,r}}}

\def\s{{\sigma}}

\def\dm{{\underline{m}}}

\def\Ak{{{\cal A }_{k}}}
\def\OT{{{\Omega_T}}}
\def\iOT{{\int_{\OT}}}

\def\iO{{\int_{\Omega}}}

\def\Qd{{Q_j^\delta}}
\def\tQd{{\widetilde{Q}_j^\delta}}

\newcommand{\wt}{\widetilde}

\newcommand{\vp}{\varphi}

\newcommand{\pa}{\partial}

\newcommand{\dv}{\mathrm{div}}

\begin{document}

\begin{frontmatter}

\title{Renormalized solutions to parabolic equations in  time and space dependent\\ anisotropic {M}usielak-{O}rlicz spaces in absence of {L}avrentiev's phenomenon}

\author[1]{Iwona Chlebicka\corref{mycorrespondingauthor}}
\cortext[mycorrespondingauthor]{Corresponding author}
\ead{i.chlebicka@mimuw.edu.pl}
\author[2]{Piotr Gwiazda}
\ead{p.gwiazda@mimuw.edu.pl}
 \author[1]{Anna Zatorska--Goldstein\fnref{myfootnote}}
  \ead{azator@mimuw.edu.pl}

\fntext[myfootnote]{The research of I.C. is supported by NCN grant no. 2016/23/D/ST1/01072.  The research of P.G. has been supported by the NCN grant  no. 2014/13/B/ST1/03094. The research of A.Z.-G. has been supported by the NCN grant  no. 2012/05/E/ST1/03232. }

\address[1]{Institute of Applied Mathematics and Mechanics,
%Faculty of Mathematics, Informatics and Mechanics,
University of Warsaw, ul. Banacha 2, 02-097 Warsaw, Poland
}
\address[2]{Institute of  Mathematics, Polish Academy of Sciences, ul. \'{S}niadeckich 8, 00-656 Warsaw, Poland
}

\begin{abstract}
We provide existence and uniqueness of renomalized solutions to a general nonlinear parabolic equation with merely integrable data on a Lipschitz bounded domain in $\rn$. Namely we study
\begin{equation*}
\left\{\begin{array}{l }
\partial_t u-\dv A(t,x,\nabla u)= f(t,x) \in L^1(\Omega_T),\\
u(0,x)=u_0(x)\in L^1(\Omega).
\end{array}\right.
\end{equation*}
The growth of the monotone vector field $A$ is assumed to be controlled by a generalized nonhomogeneous and anisotropic $N$-function $M:[0,T)\times\Omega\times\rn\to\R_+\cup\{0\} $. 

Existence and uniqueness of renormalized solutions are proven in absence of~Lavrentiev's phenomenon.  The condition we impose to ensure approximation properties of the space is a certain type of balance of interplay between the behaviour of $M$ for large $|\xi|$ and small changes of time and space variables.  Its instances  are log-H\"older continuity of variable exponent (inhomogeneous in time and space) or optimal closeness condition for powers in double phase spaces  (changing in time).

The noticeable challenge of this paper is considering the problem in non-reflexive and inhomogeneous fully
anisotropic space that changes along time. New delicate approximation-in-time result is proven and applied in the construction of renormalized solutions.

\end{abstract}

\begin{keyword}  existence of solutions \sep the Musielak-Orlicz spaces \sep parabolic problems \sep renormalized solutions
\MSC[2010] 35K55 \sep  35A01
\end{keyword}

\end{frontmatter}

%\tableofcontents

%%%%%%%%%%%%%%%%%%%%%%%%%%%%%%%%%%%%%%%%%%%%%%%%%%%%%%%%%%%%%%%%%%%%%%%%%%%%%%%
\section{Introduction}
%%%%%%%%%%%%%%%%%%%%%%%%%%%%%%%%%%%%%%%%%%%%%%%%%%%%%%%%%%%%%%%%%%%%%%%%%%%%%%%

The main result of the paper is existence and uniqueness of a renormalized solution to a general parabolic equation with merely integrable data in the spaces changing along time. Namely, we study  the problem
\[\left\{\begin{array}{ll}
\partial_t u-\dv A(t,x,\nabla u)= f(t,x) & \ \mathrm{ in}\  \OT=(0,T)\times \Omega,\\
u(t,x)=0 &\ \mathrm{  on} \ \partial\Omega,\\
u(0,x)=u_0(x) & \ \mathrm{ in}\  \Omega,
\end{array}\right.
\]
where $[0,T]$ is a finite interval, $\Omega$ is a bounded Lipschitz domain in $ \rn$, $N>1$, $f\in L^1(\OT)$, $u_0\in L^1(\Omega)$.

The~modular function $M$ controlling the growth of the operator is assumed to be inhomogeneous, i.e. changing with the position in the time-space domain $\Omega_T \subset \r^{N+1}$, and fully anisotropic, i.e. $M=M(t,x,\nabla u)$ instead of~$M=M(t,x,|\nabla u|)$. 

\medskip

 To our best knowledge, existence to parabolic problems with data below duality in spaces changing with time is not addressed yet anywhere in the literature, even in the case of $L^{p(\cdot,\cdot)}$ -- the variable exponent space with the exponent depending on time and space. We solve it for every log-H\"older continuous exponent $p:[0,T]\times\Omega\to (1,\infty)$ separated from $1$ and $\infty$. Many other interesting examples  are given in~Examples~\ref{ex:many},~\ref{ex:many:p},~\ref{ex:Ordp} and~\ref{ex:weOr}.

\medskip

Typically studies in the setting involves growth conditions on $M$ or its conjugate $M^*$ e.g.~\cite{gwiazda-ren-ell,gwiazda-ren-para,hhk,mmos2013}. We stress out that we  \textbf{do not} impose any of such conditions,
nor~any particular restriction on the growth of $M$, apart from it being an $N$-function (i.e. convex, with superlinear growth), cf.~Definition~\ref{def:Nf}. Also, by considering anisotropic $M$, we allow different growth behaviour of $M$ in different directions.

In order to relax typical growth conditions we require the balance of the asymptotic behaviour of the modular function, i.e. we describe the interplay between the behaviour of $M$ for large $|\xi|$ and small changes of $t$ and $x$ -- the appropriate conditions (cf.~conditions ($\mathcal{M}$) or ($\mathcal{M}_p$)). These conditions take very intuitive forms in the isotropic setting -- see conditions ($\mathcal{M}^{iso}$) or ($\mathcal{M}_p^{iso}$) in the Theorem \ref{theo:main0} below. Their instances are log-H\"older continuity of variable exponent or optimal closeness condition for powers in double
phase spaces. The balance condition is needed only to ensure good approximation properties of the underlying function space and it can be skipped in the pure Orlicz (possibly fully anisotropic) case, i.e. when $M=M(\xi)$.

The problems similar to~\eqref{intro:para} with $A$ depending on $\nabla u$ only and with polynomial growth are very well understood. There are countless deep results concerning the corresponding problems involving the $p$-Laplace operator, $A(t,x,\xi)=|\xi|^{p-2}\xi$, stated in the Lebesgue space setting (the modular  function is then $M(t,x,\xi)=|\xi|^p$). There is a wide range of directions in~which the polynomial growth case has been developed, including the variable exponent, Orlicz, and double-phase spaces unified.  Survey~\cite{IC-pocket} describes how they can be unified  in the framework of Musielak-Orlicz spaces and used as a setting for differential equations.

The study of nonlinear boundary value problems in~non-reflexive Orlicz-Sobolev-type setting originated in the work of Donaldson~\cite{Donaldson} and Gossez~\cite{Gossez2,Gossez3,Gossez}. We refer to the paper of Mustonen and Tienari~\cite{Mustonen} for a~summary of~the~results. The~case of~vector Orlicz spaces with an anisotropic modular function, but independent of spacial or time variables, was investigated in~\cite{Gparabolic}. 

 The Musielak-Orlicz setting in full generality has been studied systematically starting from~\cite{Musielak,Sk1,Sk2} and developed inter alia around the theory arising from fluid mechanics~\cite{gwiazda-non-newt,gwiazda-tmna,gwiazda2,Aneta}.  For other recent developments of~the framework of the spaces let us refer e.g. to~\cite{yags,HPHPAK,hhk,hht,mmos:ap,mmos2013}. Typically the research concentrates, however, mostly on the $\Delta_2/\nabla_2$-case, or -- even if without structural conditions of $\Delta_2$--type (and thus done in nonreflexive spaces) --  when the modular function was trapped between some power-type functions  usually briefly described as $p,q$-growth. This direction comes from the fundamental papers~\cite{Marc1,Marc2} by Marcellini and despite it is well understood area it is still an active field especially from the point of view of modern calculus of variations and potential theory, see e.g.~\cite{ELM,HPHPAK,EMM,EMM2,bcm17,min-double-reg1,Baroni-Riesz,KuMi-Wolff-para}. 

 However, there is a vast range of~$N$-functions that do not satisfy the $\Delta_2$ condition, e.g.
\begin{itemize}
\item  $M(t,x,\xi)=a(t,x)\left( \exp(|\xi|)-1+|\xi|\right)$;
\item  $M(t,x,\xi)= a(t,x)|\xi_1|^{p_1(t,x)}\left(1+|\log|\xi||\right)+\exp(|\xi_2|^{p_2(t,x)})-1$, when $(\xi_1,\xi_2)\in\R^2$ and $p_i:\OT\to[1,\infty]$.
    This is a model example to imagine what we mean by an anisotropic modular function.
\end{itemize}
 
  Resigning from growth restrictions requires some density properties of the space, cf.~\cite{yags,pgisazg1,pgisazg2}, which we discuss in further parts of the paper. Particularly challenging is admitting not only space inhomogeneity, but also time dependence of the modular function. To our best knowledge this issue is raised only in~\cite{para-t-weak,ASGpara}, where existence is provided for bounded-data problems. 
  
  \medskip

Partial differential equations with data not in the dual space but only integrable received special attention. The cornerstone of the theory was the work of DiPerna and Lions~\cite{diperna-lions}, where they introduced the notion of~the~renormalized solution in the context of the~Boltzmann equation. Let us also refer to fundamental developments by Boccardo, Giachetti, Diaz, and Murat~\cite{boc-g-d-m} and Murat~\cite{murat}.  Other seminal idea for problems with $L^1$-data are SOLA (solutions is obtained as a limit of approximation) coming from Boccardo and Gallou\"et~\cite{bgSOLA,bgSOLA-cpde}.  Finally, {entropy solutions} are considered starting from papers by Benilan, Boccardo, Gallou\"et, Gariepy, Pierre, and Vazqu\'ez~\cite{bbggpv}, Boccardo, Gallou\"et, and Orsina~\cite{bgo}, and Dall'Aglio~\cite{dall}. Let us stress that there are cases when the mentioned notions coincide. Indeed, in~\cite{DrP} the equivalence between entropy and renormalized solutions for problems with polynomial growth is provided. Meanwhile, the corresponding result in the variable exponent and the Orlicz settings are provided together with the proofs of~the existence of~renormalized solutions in~\cite{ZZ,zhang17}, respectively.

In the parabolic setting, renormalized solutions were studied e.g. in~\cite{B,BMR,boc-ors1,DrP,P}, in the variable exponent setting~\cite{BWZ,LG,ZZ}. For very recent results on entropy and renormalised solutions, we refer also to~\cite{ch-o,F15,MMR17,zhang17}. Parabolic problems in non-reflexive Orlicz-Sobolev spaces are studied in this context in~\cite{HMR,MMR17,R10,zhang17}, while in the nonhomogeneous and non-reflexive Musielak-Orlicz spaces in~\cite{pgisazg2,gwiazda-ren-para}.   See~\cite{IC-pocket} for deeper considerations on the problems with data below duality in various instances of Musielak-Orlicz spaces.

 In all the results mentioned, except~\cite{pgisazg2}, the $\Delta_2$ condition on $M^*$ was imposed (this entails separability of~$L_{M^*}$, see~\cite{Aneta}). Thereby, this our paper can be treated as a follow-up of~\cite{pgisazg2} relaxing the balance condition therein and admitting time-dependence of the modular function. Let us repeat that to our best knowledge there are no results on existence of solutions to parabolic problems with data below duality stated in the space changing with time, even in weighted Sobolev spaces nor in variable exponent spaces (with weight or exponent depending also on the time variable). For a wide range of examples we  see Example~\ref{ex:many}.
 
 \medskip
 
 Let us present the main objectives.

\subsubsection*{The operator}

We consider $A$~belonging to an Orlicz class with respect to the last variable. Namely, we assume that function $A:[0,T]\times\Omega\times\rn\to\rn$  satisfies the following conditions.
\begin{enumerate}[($\mathcal{A}$1)]
\item \label{A1} $A$ is a Carath\'eodory's function, i.e. it is measurable w.r. to $(t,x)\in \OT$ and continuous w.r. to $\xi$;
\item \label{A2} \textbf{Growth and coercivity.} There exists an $N$-function  $M:[0,T]\times\Omega\times\rn\to\r$ and a constant $c_A>0$ such that for all $\xi\in\rn$ we have
\[ M(t,x,\xi)\leq A(t,x,\xi)\xi  \qquad\text{and}\qquad c_A M^*(t,x,A(t,x,\xi))\leq  M(t,x,\xi),\]
where $M^*$ is conjugate to $M$ (see Appendix A for the definitions).
\item \label{A3} \textbf{Weak monotonicity.} For all $\xi,\eta\in\rn$ and $x\in\Omega$ we have
\[(A(t,x,\xi) - A(t,x, \eta)) \cdot (\xi-\eta)\geq 0.\]
\end{enumerate}
In the fully anisotropic Musielak-Orlicz setting the choice of proper functional setting is not obvious. When gradient is considered in the anisotropic space, the function itself can be assumed to belong to various different isotropic spaces.  We choose the most intuitive classical Lebesgue's space. Thus, the framework we investigate involves the functional spaces
\[
\begin{split} \VTM  &=\{u\in L^1(0,T;W_0^{1,1}(\Omega)):\ \nabla u\in L_M(\OT;\rn)\},\\
 \VTMi %& =\{u\in L^\infty(0,T; L^2(\Omega))\cap L^1(0,T;W^{1,1}_0(\Omega)):\ \nabla u\in L_M(\OT;\rn)\}=\\
 & =\VTM \cap L^\infty(0,T; L^1(\Omega)).
\end{split}
\]
For the definition of Musielak-Orlicz space $L_M$  generated by an $N$-function $M:[0,T]\times\Omega\times\rn\to\R$ we refer the reader to Section \ref{ssec:mo spaces}. Let us the concisely summarize features and difficulties of the framework of the Musielak-Orlicz spaces we consider. They are reflexive, provided $M,M^*\in\Delta_2$ close to infinity. Indeed, the variable exponent Lebesgue spaces, as well as~weighted or the double phase space are reflexive even if the weight $a$ or the exponents are only assumed to be separated from $1$, bounded, and measurable. Let us recall that various analytical difficulties are expected, when the modular function has growth far from polynomial. 

\subsection*{Absence of Lavrentiev's phenomenon}
If the growth of the modular function is arbitrary, in general the space is not reflexive, weak and weak-$*$-topology do not coincide, and  in fact yet another topology becomes to be relevant. In the case of the isotropic Orlicz spaces, according to Gossez~\cite{Gossez}, weak derivatives in the Orlicz-Sobolev spaces are strong derivatives with respect to the modular topology. Here we deal with not only with anisotropy, but also an additional difficulty resulting from inhomogeneity, namely the so-called Lavrentiev phenomenon. It occurs, when the infimum of a variational functional taken over space of smooth functions is strictly larger than the infimum over the (larger) space of functions, on which the functional is defined, see~\cite{LM}. The notion of the Lavrentiev phenomenon became naturally generalised to the situation, where functions from a certain space cannot be approximated by regular ones. This can occur in variable exponent spaces~\cite{ZV},  in the double-phase space~\cite{min-double-reg1}, as well as in the case linking them~\cite{bcm-st}.   Let us stress that kind of the Meyers-Serrin theorem, saying that weak derivatives are strong ones with respect to the modular topology, in the Musielak-Orlicz spaces holds only in absence of~Lavrentiev's phenomenon and this is the scope we work in. 

\medskip

To solve this problem the approximation of the gradients in the modular topology was proven in~\cite[Section~3]{para-t-weak} under the conditions below. 

\medskip

In the fully anisotropic case we shall consider the modular functions satisfying a balance condition.

\begin{enumerate}
\item[($\mathcal{M}$)] \label{M}  Suppose that there exists a function  $\Theta :[0,1]^2\to\rp$ nondecreasing with respect to each of the variables, such that
\begin{equation}
\label{ass:M:vp}  \limsup_{\delta\to 0^+} \Theta (\delta,\delta^{-N})<\infty ,\end{equation} which express the relation between $M(t,x,\xi )$ and \begin{equation}
\label{MIQ}M_{I,Q}(\xi ):= {\essinf}_{\substack{t\in  {I\cap[0,T]},\\ x\in Q\cap\Omega}}M(t,x,\xi ).
\end{equation} We assume that there exist $\xi_0\in\rn$ and $\delta_0>0$, such that for every interval $I\subset\r$, such that $|I|<\delta<\delta_0,$ and every cube $Q\subset \rn$ with ${\rm diam}\, Q<4\delta\sqrt{N}$
\begin{equation}
\label{ass:M:reg:IQ}
\frac{M(t,x,\xi )}{(M_{I,Q})^{**}(\xi)}
\leq \Theta  \left(\delta,  |\xi| \right)\quad\text{for a.e. }t\in I,\ \ \text{a.e. } x\in Q\cap\Omega,\ \text{ and all }\ \xi\in\rn:\ |\xi|>|\xi_{0}|,
\end{equation}
where by $(M_{I,Q})^{**}(\xi)=((M_{I,Q})^*(\xi))^* $, we denote the greatest convex minorant of the infimum from~\eqref{MIQ} (coinciding with the~second conjugate function, cf. Definition~\ref{def:conj}).
\end{enumerate}

\noindent When the modular function has at least power-type growth, we relax ($\mathcal{M}$) as follows.

\begin{enumerate}[($\mathcal{M}_p$)]
\item \label{Mp} Suppose for $\xi\in\rn$, such that $|\xi|>|\xi_p|$ \begin{equation}
\label{M>p}
M(t,x, \xi )\geq c_{gr}|\xi|^p \qquad\text{with }\quad  1<p<N\text{ and }\  c_{gr}>0 
\end{equation}
and  that there exists a function   $\Theta_p :[0,1]^2\to\rp$ nondecreasing with respect to each of the variables, such that~\eqref{ass:M:reg:IQ} holds with $\Theta$ substituted by $\Theta_p$ satisfying
\begin{equation} \label{ass:M:reg-p} \limsup_{\delta\rightarrow0^+}
\Theta_p (\delta,\delta^{-\frac{N}{p}})<\infty .
\end{equation}
\end{enumerate}
We point out that these results are optimal within some special cases (variable exponent, double phase together with its borderline case). Wider range of examples is presented below (Examples~\ref{ex:many},~\ref{ex:many:p},~\ref{ex:Ordp} and~\ref{ex:weOr}).

The condition  ($\mathcal{M}$) (resp.~($\mathcal{M}_p$))  not only takes into account time-dependence, but it is more general (admits less space-variable-control) and  is much easier to understand than its corresponding condition \cite[(M)]{pgisazg2}. Moreover, it is applied only in~the~proofs of~approximation results. For the rest of the reasoning, it suffices that $M$ is an $N$-function. The condition  ($\mathcal{M}$) (resp.~($\mathcal{M}_p$))  is also sufficient for  the new delicate approximation-in-time result provided in Section~\ref{sec:approximation} and necessary in our construction of renormalized solutions. 

\subsection*{Renormalized solutions}
We recall the definition of a renormalized solution. For this we need to introduce the symmetric truncation defined as follows\begin{equation}T_k(f)(x)=\left\{\begin{array}{ll}f & |f|\leq k,\\
k\frac{f}{|f|}& |f|\geq k.
\end{array}\right. \label{Tk}
\end{equation}
We say that a function $u$ is a \textbf{renormalized solution} to the problem
\begin{equation}\label{intro:para}\left\{\begin{array}{ll}
\partial_t u-\dv A(t,x,\nabla u)= f(t,x) & \ \mathrm{ in}\  \OT=(0,T)\times \Omega,\\
u(t,x)=0 &\ \mathrm{  on} \ \partial\Omega,\\
u(0,x)=u_0(x) & \ \mathrm{ in}\  \Omega,
\end{array}\right.
\end{equation}
where $[0,T]$ is a finite interval, $\Omega$ is a bounded Lipschitz domain in $ \rn$, $N>1$, $f\in L^1(\OT)$, $u_0\in L^1(\Omega)$, if it satisfies the following conditions:
\begin{enumerate}[($\mathcal{R}$1)]
\item \label{R1} $u\in L^1(\OT)$ and for each $k>0$  \[T_k(u)\in \VTM  ,\qquad A(\cdot,\cdot,\nabla T_k(u))\in L_{M^*}(\OT;\rn).\]
\item \label{R2} For every $h\in C^1_0(\R)$ and all $\varphi\in \VTMi$, such that $\partial_t\vp\in L^\infty(\Omega_T)$ and $\vp(\cdot,x)$ has a compact support in $[0,T)$ for a.e. $x\in\Omega$, we have
\begin{equation*}
%\label{eq:r2}
-\int_{\OT} \left(\int_{u_0(x)}^{u(t,x)}h(\sigma)d\sigma\right)  \partial_t \vp\ dx\,dt+\int_{\OT}A(t,x,\nabla u)\cdot\nabla(h(u)\vp) \,dx\,dt=\int_{\OT}f h(u)\vp \,dx\,dt.
\end{equation*}
\item \label{R3} $\displaystyle{\int_{ \{l<|u|<l+1\}}A(t,x,\nabla u)\cdot\nabla u\, dx\,dt\to 0}$ as $l\to\infty$.
\end{enumerate} 

\medskip

Our main result yields the existence of a unique renormalized solution to~\eqref{intro:para} in the \textbf{fully anisotropic} case. However, we would like to present first the more intuitive isotropic case, when $M$ is a radial function with respect to the gradient variable $\xi$, i.e. $M=M(t,x,|\xi|)$.

\begin{theo}[Isotropic case] \label{theo:main0}
Suppose $[0,T]$ is a finite interval, $\Omega$ is a bounded Lipschitz domain in $ \rn$, $f\in L^1(\OT)$, $u_0\in L^1(\Omega)$, function $A$ satisfies assumptions ($\mathcal{A}$\ref{A1})--($\mathcal{A}$\ref{A3}) with a locally integrable $N$-function~$M:[0,T]\times\Omega\times\r\to\r$. Assume further that at least one of the following assumptions holds:
\begin{enumerate}
\item[($\mathcal{M}^{iso}$)]  (arbitrary growth) there exists a function  $\Theta^{iso}:\rp^2\to\rp$ nondecreasing with respect to each of the variables, such that
    \begin{equation} \label{ass:M:Th:lim:iso}
    \limsup_{\delta\to 0^+} \Theta^{iso} (\delta, \delta^{-N})<\infty,
    \end{equation}
    and
    \begin{equation} \label{ass:M:reg:iso}
     \frac{M(t,x,s)}{M(s,y,s)}
    \leq \Theta^{iso} \left(|t-s|+c_{sp} |x-y | , s\right)
    \end{equation}

\item[($\mathcal{M}^{iso}_p$)] (at least power-type growth)
    there exists a function  $\Theta^{iso}_p:\rp^2\to\rp$, nondecreasing with respect to each of the variables, satisfying
    \begin{equation} \label{ass:M:Th:lim:iso:p}
    \limsup_{\delta\to 0^+} \Theta^{iso} (\delta, \delta^{-N/p})<\infty,
    \end{equation}   such that for all $s>s_p$
    \begin{equation} \left\{\begin{array}{l}M(t,x, s )\geq c_{gr}\,s^p \qquad\text{with }\  p>1\text{ and }\  c_{gr}>0,\\
    \label{ass:M:reg:iso:p}
    \frac{M(t,x,s)}{M(s,y,s)}
    \leq \Theta_p^{iso} \left(|t-s|+c_{sp} |x-y | , s\right).\end{array}\right.
    \end{equation} 
\end{enumerate}
Then there exists a unique renormalized  solution to the problem~\eqref{intro:para}, i.e. there exists $u$ satisfying ($\mathcal{R}$\ref{R1})--($\mathcal{R}$\ref{R3}).
\end{theo}

Our most general result reads as follows.

\begin{theo}[Fully anisotropic case]\label{theo:aniso}
Suppose $[0,T]$ is a finite interval, $\Omega$ is a bounded Lipschitz domain in $ \rn$,   $N>1$,  $f\in L^1(\OT)$, $u_0\in L^1(\Omega)$, function $A$ satisfy assumptions ($\mathcal{A}$\ref{A1})--($\mathcal{A}$\ref{A3}) with  a locally integrable $N$-function $N$-function~$M:[0,T]\times\Omega\times\rn\to\rp.$ Assume further that  $M$ satisfies the condition $(\mathcal{M})$ or $(\mathcal{M}_p)$. Then there exists a unique renormalized  solution to the problem~\eqref{intro:para}, i.e. there exists $u$, satisfying ($\mathcal{R}$\ref{R1})--($\mathcal{R}$\ref{R3}).
\end{theo}

Before a load of examples, we would like to compare this result with earlier results of the authors~\cite{pgisazg2}. The equation considered in~\cite{pgisazg2} is an analogue of~\eqref{intro:para}, but is posed in Musielak-Orlicz spaces equipped with time-independent modular function. Moreover, the balance conditions here have more general form and the retrieved approximation results hold not only under the log-H\"older condition in the variable exponent spaces, but also within the sharp range of parameters in the closeness condition in the double phase space.

Nonetheless, the construction of approximation needed in the proof is very delicate and we cannot cover here the reflexive case included in~\cite{pgisazg2} (in the space not changing with time). Therefore, we pose a question.

\begin{oq}$ $\\  Is it possible to prove Theorem~\ref{theo:aniso} under assumption $M, M^*\in\Delta^\infty_2$  instead of $(\mathcal{M})$ / $(\mathcal{M}_p)$? 
\end{oq}

Let us pass to wide range of examples within our setting.

\begin{coro}[Skipping ($\cal M$) / $(\mathcal{M}_p)$ -- Orlicz case]
In the pure Orlicz case, i.e. when \[M(t,x,\xi)=M(\xi),\] the balance conditions $(\mathcal{M})$ or $(\mathcal{M}_p)$ do  not carry any information and can be skipped. Therefore, as a direct consequence of the above theorem we get existence of renormalized solutions to parabolic problem~\eqref{eq:bound} in the anisotropic Orlicz space without growth restrictions. 
\end{coro}

Theorem~\ref{theo:bound} provides existence results in the following cases when $M$ satisfies ($\cal M$) or $(\mathcal{M}_p)$.
\begin{ex}[Classical problems under condition ($\cal M$)]\label{ex:many}
We cover in particular the following problems.
\begin{itemize}
\item  When  $M=|\xi|^{p}$ with  $1<p <\infty$, in classical Sobolev spaces ($\nabla u\in L^1(0,T;W^{1,p}_0(\Omega)$) for $p$-Laplace problem $\partial_t u -\Delta_p u = f$, we study
\[\partial_t u -\dv \big (b(t,x)|\nabla u|^{p-2}\nabla u \big)= f(t,x)\]
{with} bounded $b:\OT\to\rp$ such that $0<<b<<\infty.$
\item  When  $M=|\xi|\log^\alpha (1+|\xi|)$ in $L\log^\alpha L$ spaces  with $\alpha> 0$  we study
\[\partial_t u -\dv \Big(b(t,x)\frac{\log^\alpha({\rm e}+|\nabla u|)}{|\nabla u|}\nabla u \Big)= f(t,x)\]
{with}  bounded $b:\OT\to\rp$ such that $0<<b<<\infty.$
\item When $M=|\xi|^{p(t,x)}$ in variable exponent spaces with log-H\"older $p:\OT\to (1,\infty)$ such that $1<<p <<\infty$   we study
\[\partial_t u -\dv \big (b(t,x)|\nabla u|^{p(t,x)-2}\nabla u )= f(t,x) \]
{with} bounded $b:\OT\to\rp$ such that $0<<b<<\infty.$
\item When $M=|\xi|^p+a(t,x)|\xi|^p\log({\rm e}+|\xi|)$ in double phase spaces with mild transition, with $1<p<\infty$ and with a log-H\"older weight $a:\OT\to (1,\infty)$ and possibly touching zero; we study
\[\partial_t u -\dv \Big (b(t,x)\big(1 +a(t,x)\log({\rm e}+ |\nabla u|) \big)|\nabla u|^{p-2}\nabla u \Big)= f(t,x)\]
{where} $b:\OT\to\rp$ is bounded and such that $0<<b<<\infty.$ 

See~\cite{bcm-st} for the explanation in what sense this space  is the bordeline case between the variable exponent spaces and double-phase spaces (covered sharply under (${\cal M}_p$)).
\end{itemize}
\end{ex} 
\begin{ex}[Problems under condition (${\cal M}_p$)]\label{ex:many:p}
We cover in particular the following problems. 
\begin{itemize}
\item When $M=|\xi|^p+a(t,x)|\xi|^q$ in double phase spaces, with $1<p,q<\infty$ and a function $a:\OT\to[0,\infty)$ being such that $a\in C^{0,\alpha}(\OT)$ and possibly touching zero; we study
\[\partial_t u -\dv \Big (b(t,x)\big(|\nabla u|^{p-2}\nabla u +a(t,x) (|\nabla u|^{q-2}\nabla u\big)\Big)= f(t,x)\qquad\text{if}\qquad \frac{q}{p}\leq 1+\frac{\alpha}{N}\]
{where} $b:\OT\to\rp$ is  bounded  and such that $0<<b<<\infty.$

Note that the range of parameters is sharp for absence of Lavrentiev's phenomenon due to~\cite{min-double-reg1}.
\item  When $M(t,x,\xi)=|\xi|^{p(t,x)}+a(t,x)|\xi|^{q(t,x)}$  in variable exponent double-phase spaces  with log-H\"older $p,q:\OT\to (1,\infty)$ such that  $1<p_-<p(t,x)<q(t,x)<<\infty$ and a function $a:\OT\to[0,\infty)$ being such that $a\in C^{0,\alpha}(\OT)$ and possibly touching zero; we study
\[\partial_t u -\dv \Big (b(t,x)\big(|\nabla u|^{p(t,x)-2}\nabla u +a(t,x) (|\nabla u|^{q(t,x)-2}\nabla u\big)\Big)= f(t,x)\quad\text{if}\quad \sup_{(t,x)\in\OT}\big(q(t,x)-p(t,x)\big)\leq \frac{\alpha p_-}{N} \] 
{where} $b:\OT\to\rp$ is bounded  and such that $0<<b<<\infty.$
\end{itemize}
\end{ex} 

\noindent When the growth of $M$ is far from polynomial, the meaning of the balance condition can be illustrated by the following examples.

\begin{ex}[Orlicz double phase space without growth restrictions]\label{ex:Ordp} When $M(t,x,\xi)=M_1(\xi)+a(t,x)M_2(\xi)$, where $M_1,M_2$ are (possibly anisotropic) homogeneous $N$-functions without prescribed growth  such that $M_1(\xi)\leq M_2(\xi)$ for $\xi:$ $|\xi|>|\xi_0|$, and moreover the~function $a:\OT\to\rp$ is bounded and has a modulus of continuity denoted by $\omega_a$, we infer existence and uniqueness for solution to the problem
\[\partial_t u -\dv \left(b(t,x)\Big(\frac{M_1(\nabla u)}{|\nabla u|^2}\cdot\nabla u +a(t,x) \frac{M_2(\nabla u)}{|\nabla u|^2}\cdot\nabla u\Big)\right)= f(t,x)\in L^1(\Omega_T)\qquad\text{with}\qquad 0<<b<<\infty,\] 
provided \[\limsup_{\delta\to 0} \omega_a(\delta)\frac{\overline{M}_2(\delta^{-N})}{\underline{M}_1(\delta^{-N})}<\infty,\]
where $\underline{M}_1(s):=\inf_{\xi:\,|\xi|=s}{M}_1(\xi)$ and $\overline{M}_2(s):=\sup_{\xi:\,|\xi|=s}{M}_2(\xi)$,
or -- when $M_1$ has at least power growth -- provided\[\limsup_{\delta\to 0} \omega_a(\delta)\frac{\overline{M}_2(\delta^{-N/p})}{\underline{M}_1(\delta^{-N/p})}<\infty.\]
\end{ex}

 \begin{ex}[Weighted Orlicz spaces without growth restrictions]\label{ex:weOr}
 \rm If $M$ has a form %\begin{equation} \label{M:spec}
\[M(t,x,\xi)=\sum_{i=1}^jk_i(t,x)M_i(\xi)+M_0(t,x,|\xi|),\quad j\in\N,\]
%\end{equation}
instead of ($\mathcal{M}$) we assume only that $M_0$ satisfies~($\mathcal{M}^{iso}$), all $M_i$ for $i=1,\dots,j$ are $N$-functions and all $k_i$ are positive and satisfy $\frac{k_i(t,x)}{k_i(s,y)}\leq C_i \Theta_i(|t-s|+c_{sp} |x-y | )$ with $C_i>0$ and $\Theta_i:\rp\to\rp$ and $\Theta_i\in L^\infty$ for $i=1,\dots,j$. Then, according to computations in Appendix, we get that $M$ satifies ($\mathcal{M}$)  when we take  \[\Theta(r, s)=\sum_{j=1}^k \Theta_j(r)+\Theta_0(r,s)\qquad\text{with}\qquad \limsup_{\delta\to 0^+}\Theta(\delta,\delta^{-N})<\infty\] In the case of~($\mathcal{M}_p$) we expect $\limsup_{\delta\to 0^+}\Theta(\delta,\delta^{-N/p})<\infty$.
\end{ex}

\subsection*{The methods }

We use the framework developed in~\cite{pgisazg1,pgisazg2,gwiazda-ren-ell,gwiazda-ren-cor,gwiazda-ren-para}, where elliptic and parabolic problems in the Musielak--Orlicz spaces were studied, and apply the results of~\cite{para-t-weak}.
Since in general $M^*\not\in\Delta_2$, the understanding of~the dual pairing is not intuitive. Indeed, $A(\cdot,\cdot,\nabla(T_k(u)))$ and $\nabla(T_k(u))$ do not belong to the dual spaces. Relaxing growth condition on the modular function restricts the admissible classical tools, such as the Sobolev embeddings,~the~Rellich-Kondrachov compact embeddings, or~Aubin-Lions Lemma (applied in~\cite{gwiazda-ren-para} to~prove almost everywhere convergence). The proof of the existence of a renormalized solution involves the classical truncation ideas, the Young measures methods and monotonicity arguments. Uniqueness results from the comparison principle.

The scheme follows the ideas of~\cite{pgisazg2}. First, we establish certain types of convergence of truncations of~solutions $T_k(u_n)$ (Proposition~\ref{prop:convTk}). Then, the radiation control condition ($\mathcal{R}$3)  for $u_n$ is provided (Proposition~\ref{prop:contr:rad:n}). Next, we apply the comparison principle to obtain almost everywhere convergence of $u_n$. We identify $A(t,x,\nabla T_k(u))$ as the weak-* limit in $L_{M^*}$ of~$A(t,x,\nabla T_k(u_n))$ (Proposition~\ref{prop:convTkII}). Finally we conclude the proof of~existence of~renormalized solutions. Weak $L^1$-convergence of $A(t,x,\nabla T_k(u_n))\cdot \nabla T_k(u_n)$ is obtained using the Young measures.

Since the modular function is time-dependent, the identification of limits of approximate sequences is highly non-trivial. The space we deal with is, in~general, neither separable, nor reflexive. The lack of~precise control on the growth of~$A$ together with the merely integrable right-hand side cause noticeable difficulties in studies on convergence of approximation. The construction of our renormalized solutions holds in~the~absence of~Lavrentiev's phenomenon, i.e. when the functions from the relevant space can be approximated by smooth ones.  The critical place, where this paper differs from~\cite{pgisazg2}, is that time-dependence of the modular function essentially complicates the construction of time-approximation, since the Landes regularization previously used in the corresponding study stops to converge modularly.    Thus, as a tool we need to provide new results on approximation having by far more delicate properties, see  Theorem~\ref{theo:approx-t}. Careful merging the ideas of Landes on the splitted time-interval combined with analysis of concentration of density of the mollifier reaches the point.

\subsection*{Organization of the paper}
The paper is organized as follows. Section~\ref{sec:frame} introduces notation, basic information on the spaces, as well as it recalls some results of~\cite{para-t-weak}  necessary in our considerations including integration-by-parts fomula, comparison principle, monotonicity trick, and existence of weak solutions to bounded-data problem. Section~\ref{sec:approximation} is devoted to the main tool we derive in the paper, namely time-approximation. In Section~\ref{sec:main proof} we present the proof of the main results. Some classical definitions and theorems are listed in Appendix.

%%%%%%%%%%%%%%%%%%%%%%%%%%%%%%%%%%%%%%%%%%%%%%%%%%%%%%%%%%%%%%%%%%%%%%%%%%%%%%%%
\section{Analytical framework} \label{sec:frame}

In this section we provide necessary notation and basic information on Musielak-Orlicz spaces, afterwards we give also formulations of results coming from~\cite{para-t-weak} such as integration-by-parts formula, comparison principle, monotonicity trick, existence result to the problem with bounded data, and lemma on simplification of anisotropic conditions in the isotropic situation.

\subsection{Notation } We assume $\Omega\subset\rn$ is a bounded Lipschitz domain, $\Omega_T=(0,T)\times\Omega,$ $\Omega_\tau=(0,\tau)\times\Omega.$ If $V\subset\R^K$, $K\in\N$, is a bounded set, then $C_c^\infty(V)$ denotes the class of smooth functions with support compact in $V$. We denote positive part of function signum by $\sg(s)=\max\{0,s/|s|\}$. 

We note that according to \cite[Lemma~2.1]{bbggpv}, for every $u\in W^{1,1}(\Omega)$, there exists a unique measurable function $Z_u:\Omega\to\rn$ such that\[\nabla (T_t(u))=\chi_{\{|u|<t\}}Z_u\quad\text{a.e. in } \Omega, \text{ for every }{t>0}.
\]
Thus, in the theory $Z_u$ is called the generalized gradient of~$u$. Abusing slightly the notation, for $u$ with locally integrable $Z_u$, it is written simply $\nabla u$ instead.

\subsection{Musielak-Orlicz spaces} \label{ssec:mo spaces}

\begin{defi}[$N$-function]\label{def:Nf} Suppose $\Omega\subset\rn$ is an open bounded set. A~function   $M:[0,T]\times\Omega\times\rn\to\r$ is called an $N$-function if it satisfies the
following conditions:
\begin{enumerate}
\item $ M$ is a Carath\'eodory function (i.e. measurable with respect to $(t,x)\in\OT$ and continuous with respect to the last variable), such that $M(t,x,0) = 0$, $\essinf_{(t,x)\in\OT}M(t,x,\xi)>0$ for $\xi\neq 0$, and $M(t,x,\xi) = M(t,x, -\xi)$ a.e. in $\Omega$,
\item $M(t,x,\xi)$ is a convex function with respect to $\xi$,
\item $\lim_{|\xi|\to 0}\mathrm{ess\,sup}_{(t,x)\in\OT}\frac{M(t,x,\xi)}{|\xi|}=0$,
\item $\lim_{|\xi|\to \infty}\mathrm{ess\,inf}_{(t,x)\in\OT}\frac{M(t,x,\xi)}{|\xi|}=\infty$.
\end{enumerate}
Moreover, we call $M$ a locally integrable $N$-function if additionally for every measurable set $G\subset\OT$ and every $z\in\rn$ it holds that
\begin{equation}
\label{ass:M:int}\int_G M(t,x,z)\,dx\,dt<\infty.
\end{equation}
\end{defi}

\begin{defi}[Complementary function] \label{def:conj}
The complementary~function $M^*$ to a function  $M:[0,T]\times\Omega\times\rn\to\r$ is defined by
\[M^*(t,x,\eta)=\sup_{\xi\in\rn}(\xi\cdot\eta-M(t,x,\xi)),\qquad \eta\in\rn,\ x\in\Omega.\]
If $M$ is an $N$-function and $M^*$ its complementary, we have the Fenchel-Young inequality
\begin{equation}
\label{inq:F-Y}|\xi\cdot\eta|\leq M(t,x,\xi)+M^*(t,x,\eta)\qquad \mathrm{for\ all\ }\xi,\eta\in\rn\mathrm{\ and\ a.e.\ }(t,x)\in\OT.
\end{equation}
\end{defi}
%\begin{rem}\label{rem:f*<g*}If $f(x,\xi)\leq g(x,\xi)$, then $g(x,\xi)^*\leq f^*(x,\xi)$.\end{rem}

\begin{rem}\label{rem:2ndconj} For any function $f:\r^M\to\r$ the second conjugate function $f^{**}$ is convex and $f^{**}(x)\leq f(x)$. In fact,  $f^{**}$ is a convex envelope of $f$, namely it is the biggest convex function smaller or equal to~$f$.
\end{rem}

\begin{defi}\label{def:MOsp} Let $M$ be a locally integrable $N$-function. We deal with the three  Orlicz-Musielak classes of functions.\begin{itemize}
\item[i)]${\cal L}_M(\OT;\rn)$  - the generalised Orlicz-Musielak class is the set of all measurable functions\\ $\xi:\OT\to\rn$ such that
\[\int_\OT M(t,x,\xi(t,x))\,dx\,dt<\infty.\]
\item[ii)]${L}_M(\OT;\rn)$  - the generalised Orlicz-Musielak space is the smallest linear space containing ${\cal L}_M(\Omega;\rn)$, equipped with the Luxemburg norm
\[||\xi||_{L_M}=\inf\left\{\lambda>0:\int_\OT M\left(t,x,\frac{\xi(t,x)}{\lambda}\right)\,dx\leq 1\right\}.\]
\item[iii)] ${E}_M(\OT;\rn)$  - the closure in $L_M$-norm of the set of bounded functions.
\end{itemize}
\end{defi}
Then
\[{E}_M(\OT;\rn)\subset {\cal L}_M(\OT;\rn)\subset { L}_M(\OT;\rn),\]
where the inclusions can be strict.

\medskip

The space ${E}_M(\OT;\rn)$ is separable and due to~\cite[Theorem~2.6]{Aneta} the following duality holds \[({E}_M(\OT;\rn))^*=L_{M^*}(\OT;\rn).\]

\medskip

 We say that an $N$-function $M:[0,T]\times\Omega\times\rn\to\r$ satisfies $\Delta_2$ condition close to infinity (denoted $M\in\Delta_2^\infty$) if there exists a constant $c>0$ and nonnegative integrable function $h:\OT\to\r$ such that for a.e. $(t,x)\in\OT$ it holds 
\begin{equation*}
%\label{D2}
 M(t,x,2\xi)\leq cM(t,x,\xi)+h(t,x)\qquad\text{for all}\quad \xi\in\rn:\ \ |\xi|>|\xi_0|.
\end{equation*}
If $M\in\Delta_2^\infty$, then
\[{E}_M(\OT;\rn)= {\cal L}_M(\OT;\rn)= {L}_M(\OT;\rn)\]
and $L_M(\OT;\rn)$ is separable. When both  $M,M^*\in\Delta^\infty_2$  then $L_M(\OT;\rn)$ is reflexive, see~\cite{GMWK,gwiazda-non-newt}. 

\medskip

We face the problem \textbf{without} this structure.

\medskip

We apply the following modular Poincar\'{e}-type inequality.
\begin{theo}[Modular Poincar\'e inequality,~\cite{CGZG}]\label{theo:Poincare}
Let $B:\rp\to\rp$ be an arbitrary Young function,  $\Omega\subset\rn$ be a bounded Lipschitz domain, and $W^{1,B}(\Omega)=\{u\in L^1(\Omega):\ \nabla u\in L_B(\Omega)\}$. Then there exist $c^1_P,c^2_P>0$ dependent on $\Omega$ and $N$, such that for every $g\in W_{loc}^{1,1}(\OT)$, such that $\int_\OT B(|\nabla g|)\,dx\,dt<\infty$ and $g$ belongs to weak-* closure of $C_0^\infty(\Omega)$ in $W^{1,B}(\Omega)$ we have
\[\int_\OT B(c^1_P|g|)\,dx\,dt\leq c^2_P
\int_\OT B(|\nabla g|)\,dx\,dt.\]
\end{theo}

%%%%%%%%%%%%%%%%%%%%%%%%%%%%%%%%%%%%%%%%%%%%%%%%%%%%%%%%
\subsection{Auxiliary results}\label{ssec:pre}
%%%%%%%%%%%%%%%%%%%%%%%%%%%%%%%%%%%%%%%%%%%%%%%%%%%%%%%%
We provide here formulations of the results of~\cite{para-t-weak} by the authors, which are necessary in our considerations.

\subsubsection*{Integration-by-parts formula}

Arbitrary growth of the modular function implies that the following result is not direct, while inhomogeneity requires absence of Lavretiev's phenomenon.

\begin{theo}[Proposition~4.1, \cite{para-t-weak}]\label{theo:intbyparts} Suppose  a locally integrable $N$-function $M$ satisfies assumptions ($\mathcal{M}$) (resp.~($\mathcal{M}_p$)). Suppose $u:\Omega_T\to\r$ is a measurable function such that for every $k\geq 0$, $T_k(u)\in \VTM$, $u(t,x)\in L^\infty([0,T];L^1(\Omega))$. Let us assume that there exists $u_0\in L^1(\Omega)$ such that $u_0(x):=u(0,x)$. Furthermore, assume that there exist $A\in L_{M^*}(\Omega_T;\rn)$  and $F\in L^1(\Omega_T)$ satisfying
\begin{equation}
\label{eq:theo:int-by-parts-1}
-\int_{\OT}(u-u_0)\partial_t \vp \,dx\,dt+\int_{\OT}A\cdot \nabla\vp \,dx\,dt=\int_{\OT}F\, \vp \,dx\,dt,\qquad \forall_{\vp\in{C_c^\infty}([0;T)\times \Omega)}.
\end{equation}
Then
\begin{equation*}
-\int_{\OT} \left(\int_{u_0}^u h(\s)d\s\right) \partial_t \xi \ \,dx\,dt+\int_{\OT}A\cdot \nabla (h(u)\xi) \,dx\,dt=\int_{\OT}F h(u)\xi \,dx\,dt
\end{equation*}
holds for all $h\in W^{1,\infty}(\r)$, such that $\supp (h')$ is compact and all $\xi\in \VTMi$, such that $\partial_t\xi\in L^\infty(\OT)$ and $\supp\xi(\cdot,x)\subset[0,T)$ for a.e. $x\in\Omega$, in particular for  $\xi \in C_c^\infty([0,T)\times\overline{\Omega})$.
\end{theo}

\subsubsection*{Comparison principle}

When we notice that, according to the definition, renormalized solutions verify assumptions of~\cite[Theorem~4.1]{para-t-weak}, we get the following comparison principle for them.

\begin{theo}\label{prop:comp-princ}
Suppose  a locally integrable $N$-function $M$ satisfies assumptions ($\mathcal{M}$) (resp.~($\mathcal{M}_p$)), function $A$ satisfies assumptions ($\mathcal{A}$\ref{A1})--($\mathcal{A}$\ref{A3}). Let $v^1,v^2$ be renormalized solutions  to\[\left\{\begin{array}{l}
v^1_t-{\dv }A(t,x,\nabla v^1)= f^1\in L^1(\OT),\\
v^1(0,x)=v^1_0(x)\in L^1(\Omega)
\end{array}\right.\qquad \left\{\begin{array}{l}
v^2_t-{\dv }A(t,x,\nabla v^2)= f^2\in L^1(\OT),\\
v^2(0,x)=v^2_0(x)\in L^1(\Omega),
\end{array}\right.\]where $f^1\leq f^2$ a.e. in~$\OT$ and $ v^1_0 \leq v^2_0$ in~$\Omega.$ Then $v^1\leq v^2$ a.e. in~$\OT$.
\end{theo}

\subsubsection*{Monotonicity trick}
As a consequence of weak monotonicity, we will be able to identify some limits using the following monotonicity trick applied e.g. in~\cite{pgisazg1,pgisazg2,gwiazda-ren-ell,Aneta}.
\begin{lem}[Lemma~6.5, \cite{para-t-weak}]\label{lem:mon}
Suppose $A$ satisfies conditions ($\mathcal{A}$\ref{A1})--($\mathcal{A}$\ref{A3}) with an $N$-function $M$. Assume further that there exist \[{\cal A}\in L_{M^*}(\OT;\rn)\quad\text{ and }\quad \xi\in L_{M}(\OT;\rn),\] such that
\begin{equation}
\label{anty-mon}
\int_\OT \big({\cal A} -A(t,x,\eta)\big)\cdot(\xi(t,x) -\eta)\,dx\,dt\geq0\qquad\forall_{\eta\in L^\infty(\Omega,\rn)}.
\end{equation}
Then
\[A(t,x,\xi)={\cal A}\qquad\text{a.e. in }\ \OT.\]
\end{lem}

\subsubsection*{Existence of a weak solution for the problem with cut-off data} \label{sec:bounded}
%%%%%%%%%%%%%%%%%%%%%%%%%%%%%%%%%%%%%%%%%%%%%%%%%%%%%%%%%%%%%%%%%%%%%%%%%%%%%
The following theorem is a direct consequence of~\cite[Theorem 1.2]{para-t-weak} when we just choose there particular form of bounded data.

\begin{theo}
\label{theo:bound} Suppose $[0,T]$ is a finite interval, $\Omega$ is a bounded Lipschitz domain in $ \rn$,   $N>1$,  $f\in L^1(\OT)$, $u_0\in L^1(\Omega)$, function $A$ satisfy assumptions ($\mathcal{A}$\ref{A1})--($\mathcal{A}$\ref{A3}) with an $N$-function~$M$ satisfying ($\mathcal{M}$) or ($\mathcal{M}_p$). Then for every  $n\in\N$ there exists a weak solution to the problem  \begin{equation}\label{eq:bound}
\left\{\begin{array}{ll}
\partial_t u_n -\dv A (t,x,\nabla u_n )= T_n (f) & \ \mathrm{ in}\  \OT,\\
u_n(t,x)=0 &\ \mathrm{  on} \ (0,T)\times\partial\Omega,\\
u_n (0,\cdot)=u_{0,n}(\cdot)=T_n(u_0)
 & \ \mathrm{ in}\  \Omega.
\end{array}\right.
\end{equation} Namely, there exists $u_n \in \VTMi$, such that for any $\vp\in C_c^\infty([0,T)\times \Omega)$
\begin{eqnarray}\label{weak-reg}
-\iOT u_n \partial_t \vp\, dx\, dt- \iO u_n(0)\vp(0)\, dx+\iOT A (t,x,\nabla u_n )\cdot \nabla \vp \, dx\, dt= \iOT T_n(f)\vp\, dx\, dt.
\end{eqnarray}
\end{theo}

%%%%%%%%%%%%%%%%%%%%%%%%%%%%%%%%%%%%%%%%%%%%%%%%%%%%%%%%%%%%%%%%%%%%%%%%%%%%%%%%
\section{Approximation in Musielak-Orlicz spaces}\label{sec:approximation}
%%%%%%%%%%%%%%%%%%%%%%%%%%%%%%%%%%%%%%%%%%%%%%%%%%%%%%%%%%%%%%%%%%%%%%%%%%%%%%%%

In the case of~classical Orlicz spaces, the crucial density result was provided by Gossez~\cite{Gossez}. In the case of~$x$--dependent and anisotropic log-H\"{o}lder continuous modular functions the absence of Lavrentiev's phenomenon was proven in~\cite{pgisazg1,pgisazg2}, further refined in isotropic case in~\cite{yags} to cover both -- log-H\"{o}lder condition in the variable exponent and closeness of parameters in double-phase space sharp due to~\cite{min-double-reg1}. Finally, in~\cite{para-t-weak}  space-approximation and easy time-approximation results we need here are provided under our anisotropic conditions. 

%\begin{theo}[Approximation in space, Theorem~3.1,~\cite{para-t-weak}]\label{theo:approx-sp}
%Let $\Omega$ be a bounded Lipschitz domain in $\rn$ and a locally integrable $N$-function $M$~satisfy condition ($\mathcal{M}$) (resp.~($\mathcal{M}_p$)). Then for any $\vp$ such that $\vp\in \VTMi$  there exists a sequence $\{\vp_\delta\}_{\delta>0}\subset L^\infty (0,T; C_c^\infty(\Omega))$, such that  $ \vp_\delta\to  \vp$ strongly in $L^1(\Omega_T)$ and $\nabla\vp_\delta\xrightarrow[]{M}\nabla \vp$ in $L_M(\Omega_T;\rn)$ when $\delta\to 0$. Moreover, there exists $c=c(\Omega)>0$, such that $\|\vp_\delta\|_{L^\infty(\Omega)}\leq c \|\vp \|_{L^\infty(\Omega)}$.
%\end{theo}

The above mentioned approximation results are applied in order to obtain existence of solutions to problem with bounded data (cf. Theorem~\ref{theo:bound}), and  -- consequently -- in our considerations, although we do not use them here explicitly. However, here they are not sufficient in identification of~the~limit in Proposition~\ref{prop:convTkII}. We shall need there much more delicate  time-approximation that converges modularly, commutes with the space gradient, and has properly convergent time derivatives. When the modular function is time-dependent we cannot use the Landes regularization coming from~\cite{Landes-meth}, as it was done in~\cite{gwiazda-ren-para,pgisazg2}, because due to time-dependence of $M$ the Landes regularization is not mapping $L_M$ into itself anymore. Moreover, we shall need a~few more delicate properties here. Nonetheless, careful merging the ideas of Landes on the small but not uniformly controlled time-intervals enables to prove the following result.

\begin{theo}[Approximation in time]\label{theo:approx-t}

Let $\Omega$ be a bounded Lipschitz domain in $\rn$, a locally integrable $N$-function $M$~satisfy condition ($\mathcal{M}$) (resp.~($\mathcal{M}_p$)), $\vp\in \VTM$, and $\vp_0\in L^\infty(\Omega)$. Then there exist sequences $\{\vp_\mu\}_\mu,\{ \vp^\bullet_\tm\}_\tm\subset \VTM$ with $\mu>>1$, such that
\begin{itemize}
\item[i)] for every $\mu$ and a.e. $x\in\Omega$ function $\vp_\mu(\cdot,x)\in C^\infty([0,T))$ and satisfies
\begin{equation}
\label{vp:ode}\begin{array}{cc}
\left\{\begin{array}{ll}
\pa_t\vp_{\mu}&={\mu}(\vp -\vp_\mu)\ \ \text{a.e. in }\OT,\\
\vp_\mu(0,x)&=\vp_0(x) \quad \text{a.e. in }\Omega,
\end{array}\right.\end{array}
\end{equation}
\item[ii)] for every $\mu$ we have $\vp_\mu^\bullet(0,x)=\vp_0(x)(1-e^{-\log^2\mu})$,
\item[iii)] $(\nabla \vp)^\bullet_\tm=\nabla(\vp^\bullet_\tm)$, 
\item[iv)] $\lim_{\tm\to\infty}\vp^\bullet_\tm=  {\vp}$ strongly in $L^{1}(\Omega_T)$ and $\lim_{\tm\to\infty}(\nabla \vp )^\bullet_\tm= \nabla  {\vp}$ modularly in $L_M(\Omega_T;\rn).$

\item[v)] If additionally $\vp\in L^\infty(\OT)$, then $\|\vp^\bullet_\tm\|_{L^\infty(\OT)}\leq \|\vp \|_{L^\infty(\OT)}$,  for every $\mu$ and a.e. $x\in\Omega$ function $\vp_\mu^\bullet(\cdot,x)$ belongs to $ W^{1,\infty}([0,T))$ and furthermore \[\lim_{\mu\to\infty}\|\vp_\mu-\vp_\mu^\bullet\|_{L^\infty(\Omega_T)}=0\qquad \text{and} \qquad  \lim_{\mu\to\infty}\|\pa_t\big(\vp_\mu-\vp_\mu^\bullet\big)\|_{L^\infty(\OT)}=0.\]

\end{itemize} 
\end{theo}

 Let us consider $\xi:\R\times\Omega\to\R^N$. When $\vr_{\tm}(s)={\tm} e^{-{\tm} s}\mathds{1}_{[0,\infty)}(s)$, ${\tm}>2$, the regularized function $\xi_{\tm}:\R\times\Omega\to\R$ is defined by $\xi_{\tm}(t,x):=(\vr_{\tm}*{\xi})(t,x),$ where $*$ stands for the convolution is in the time variable. Then
\begin{equation}
\label{ximu} \xi_{\tm}(t,x) ={\tm}\int_{-\infty}^t e^{{\tm}(s-t)}\xi(s,x)\,ds.
\end{equation}
Define further
 \begin{equation}
\label{xibmu} \xi_{\tm}^\bullet(t,x) ={\tm}\int_{t-\ve(\tm)}^t e^{{\tm}(s-t)}\xi(s,x)\,ds\qquad\text{with}\qquad\ve(\tm)=
\frac{\log^2\tm}{\tm}.
\end{equation}

We provide a uniform estimate in the following lemma and then conclude the proof of Theorem~\ref{theo:approx-t}.

\begin{lem}\label{lem:step2prevt} Let an $N$-function $M$ satisfy assumptions ($\mathcal{M}$), resp.~($\mathcal{M}_p$). We extend arbitrary $\xi\in \VTMi$ by $\xi(0,x)$ on $(-\infty,0)$ and by $0$ on $(T,\infty)$ and consider $\xi^\bullet_{\tm}$ be given by~\eqref{xibmu}. Then, there exist  constants  $C_1,C_2>0$ independent of ${\tm}$, such that for all  ${\tm}>>1$ and every $\xi\in\VTMi$ we have
\begin{equation}
\label{in:Md<M-t}
\int_\OT M(t,x,\xi_{\tm}^\bullet(t,x))\, dx\,dt\leq C_1 \int_\OT M\left(t,x, C_2\xi(t,x))\right)\, dx\,dt.
\end{equation}
\end{lem}

\begin{proof} We fix an arbitrary ${\tm}>>1$ and consider a family of time intervals $I^{\frac{1}{\tm}}_i=[t_i^{\frac{1}{\tm}},t_{i+1}^{\frac{1}{\tm}})$ covering $[0,T]$ and such that $|I^{\frac{1}{\tm}}_i|=|I^{\frac{1}{\tm}}_k|=\frac{1}{\tm}$ for every $i,k>1$. We denote
\[ J^{\frac{1}{\tm}}_i:=[t^{\frac{1}{\tm}}_i,t^{\frac{1}{\tm}}_{i+1}+\ve(\mu))\qquad\text{and} \qquad |J^{\frac{1}{\tm}}_i|<\frac{1}{\mu}+\ve(\mu)=:{\nu(\mu)}\]
and observe that $\lim_{\mu\to\infty}\nu(\mu)=0$. We consider
\begin{equation}\label{Mimu}
 {M}_{i,\frac{1}{\tm}}(x,\eta)=\inf\{ M(t,x,\eta):\ \ {t\in J_i^\frac{1}{\tm}\cap[0,T]}\}, 
\end{equation} and its conjugate $(M_{i,\frac{1}{\tm}})^{**} $, see~Remark~\ref{rem:2ndconj}.  Since $M(t,x,\xi)=0$ whenever $\xi=0$, we have
\begin{equation}
\label{M:div-mult-t}\begin{split}
 \iOT M(t,x, \xi^\bullet_{\tm} (t,x)))\,dx\,dt&= \sum_{i=1}^{N^t_{1/\tm} }\int_{\Omega} \int_{I_i^\frac{1}{\tm}} M(t,x,  \xi^\bullet_{\tm} (t,x))\, dt\, dx=\\&=\sum_{i=1}^{N_{1/\tm} ^t}\int_{\Omega} \int_{I_i^\frac{1}{\tm}}  \frac{M(t,x, \xi_{\tm}^\bullet (t,x))}{(M_{i,\frac{1}{\tm}})^{**}(x, \xi_{\tm}^\bullet (t,x))}{(M_{i,\frac{1}{\tm}})^{**}(x,\xi^\bullet_{\tm} (t,x))}\,dt\,dx.\end{split}
\end{equation}

We start with estimating from above the fraction on the right--hand side in the previous display by a constant essentially using the balance condition ($\mathcal{M}$), resp.~($\mathcal{M}_p$). For this we note that then for every $i\in\{1,\dots,N^t_{1/\tm} \}$ there exists $k\in\{1,\dots,N^t_{\nu(\mu)}\}$, such that we have 
\[J^{\frac{1}{\tm}}_i\subset {\widehat{I}}^{\nu({\tm})}_{i} \]
with ${\widehat{I}}^{\nu({\tm})}_{i}$ having the properties as $I^\frac{1}{\mu}_i$ described above, but with length $\delta=\nu(\tm)$ and coming from some other division of $[0,T]$ chosen for each $i$.

Moreover, we introduce notation for a family of $N$-dimensional cubes covering the set $\Omega$. Namely, a~family $\{\Qd\}_{j=1}^{N_\delta}$ consists of~closed cubes of edge $2\delta$, such that  $\mathrm{int}\Qd\cap\mathrm{int} Q^\delta_i=\emptyset$ for $i \neq j$ and $\Omega\subset\bigcup_{j=1}^{N_\delta}\Qd$. Moreover, for each cube $\Qd$ we define the cube $\tQd$ centered at the same point and with parallel corresponding edges of length $4\delta$. Denote \[\widehat{M}_{k,j}^{\nu({\tm})}(\xi):=\inf\{M(t,x,\eta):\ \ {t\in {\widehat{I}}^{\nu({\tm})}_{k}\cap[0,T]},\ x\in\wt{Q}_j^{\nu({\tm})}\}.\] Then for a.e. $x\in\wt{Q}_j^{\nu({\tm})}$, $j=1,\dots,N_{1/\tm}$,  and $i=1,\dots,N^t_{1/\tm}$ also
$\widehat{M}_{k,j}^{\nu({\tm})}(\xi)\leq M_{i,\frac{1}{\tm}}(x,\xi)$. Therefore,
\begin{equation}\label{mimij}
  \frac{M(t,x, \eta)}{(M_{i, \frac{1}{\tm}})^{**}(x,\eta)}\leq\frac{M(t,x,\eta)}{(\widehat{M}_{k,j}^{\nu({\tm})})^{**}(\eta)}\leq\Theta\big(\nu({\tm}),|\eta| \big).
\end{equation}
Note that for every $x\in\Omega$ we can choose the cube including it and having the properties needed for the above estimate and the final estimate is uniform with respect to $x$.

On the other hand, since without loss of generality it can be assumed that $\|\xi  \|_{L^\infty(0,T;L^1(\Omega))}\leq 1$, we have
\begin{equation}
\label{xidest}\begin{split}|\xi_{\tm}^\bullet (t,x)|&
\leq   {\tm}\int_{t-\ve(\mu)}^t e^{{\tm}(s-t)} \left|{\xi  (s,x)  }\right|ds \leq  c(\Omega)\tm\|\xi  \|_{L^\infty(0,T;L^1(\Omega))}\leq c(\Omega)\tm.\end{split}
\end{equation}

Then for every $ x \in\Omega$ we can choose a cube $\wt{Q}_j^{\nu(\tm)}$ including $x$. Then, using~\eqref{mimij} and (${\cal M}$), we realize that  for arbitrary $ t \in I_i^\frac{1}{\tm}$ we get
\begin{equation}
\label{M/M<c-t}\frac{M(t,x, \xi^\bullet_{\tm} (t,x))}{(M_{i,\frac{1}{\tm}})^{**}(x,\xi_{\tm}^\bullet (t,x))}\leq  \Theta\left(\nu({\tm}), |\xi_{\tm}^\bullet (t,x)|\right)\leq \Theta\Big(\nu(\tm ), c(\Omega)\tm\Big),
\end{equation} 
where the last inequality is justified by\eqref{M/M<c-t}. Note that we get a bound uniform with respect to $x$.

Let us note that $\Theta$ is nondecreasing with respect to the first variable and thus 
\[\limsup_{\mu\to\infty}\Theta(\nu(\mu),c(\Omega)\mu)=\limsup_{\delta\to0}\Theta( (1+\log^2(c(\Omega)\delta^N))c(\Omega)\delta^N, \delta^{-N}).\]
For all $\delta<\delta_0(N)$ we have
\[\Theta( (1+\log^2(c(\Omega)\delta^N))c(\Omega)\delta^N, \delta^{-N}) \leq \Theta(  \delta , \delta^{-N})<c<\infty,\]
where the last estimate holds due to $(\cal M )$. Analogously, in the case of  $({\cal M}_p )$ by the same arguments we get also $\limsup_{\mu\to\infty}\Theta_p(\nu(\mu),c(\Omega)\mu)<c<\infty$. In any case, we can estimate the right-hand side of~\eqref{M/M<c-t} by $c$ over a cube $\wt{Q}_j^{\nu(\tm)}$. Therefore, in~\eqref{M:div-mult-t} we have%\footnote{to nie jest misprint: $Q_j^{\frac{1}{\tm}}$ nie $\wt{Q}_j^{\nu(\mu)}$ }
\begin{equation}
\label{M:div-mult-t-cubes}\begin{split}
 \iOT M(t,x, \xi^\bullet_{\tm} (t,x)))\,dx\,dt &\leq  c  \sum_{i=1}^{N_{1/\tm} ^t}\int_{\Omega} \int_{I_i^\frac{1}{\tm}}   {(M_{i,\frac{1}{\tm}})^{**}(x,\xi^\bullet_{\tm} (t,x))}\,dt\,dx.\end{split}
\end{equation} 
 
Using the above estimate in~\eqref{M:div-mult-t}, the Jensen inequality (with intrisctic constant $c_J(\mu)=1/(1-e^{-\mu\ve(\mu)})\leq 1/(1-e^{-1})=c_J(1)=:C_2$), the fact that the second conjugate is (the greatest convex) minorant, and the Young inequality for convolution, we obtain
\[
\begin{split}
 \iOT M(t,x,\xi^\bullet_{\tm}(t,x))\,dx\,dt &\leq
c\sum_{i=1}^{N_\tm^t}\iO \int_{I_i^\frac{1}{\tm}}(M_{i,\frac{1}{\tm}})^{**}
 \left(x,{\tm}
 \int_{t-s-\ve(\mu)}^{t-s}
 e^{{\tm}(s-t)}   \xi(s,x)\,ds
 \right)\,dt\,dx\\
  &\leq  c \sum_{i=1}^{N_\mu^t}\iO \int_{I_i^\frac{1}{\tm}}{\tm}
  \int_{t-s -\ve(\mu)}^{t-s}
 e^{{\tm}(t-s)}(M_{i,\frac{1}{\tm}})^{**}
 \left(x, c_J(\mu)\xi(s,x)\right)\,ds\,dt\,dx\\
 &=  c \sum_{i=1}^{N_\mu^t}\iO \int_{I_i^\frac{1}{\tm}}{\tm}
  \int_{t-s -\ve(\mu)}^{t-s}
 e^{{\tm}(t-s)} M
 \left(t, x, c_J(1) \xi(t,x)\right)\,ds\,dt\,dx\\
 &\leq  c\sum_{i=1}^{N_\mu^t}\iO \|\tm
 e^{{\tm}\cdot}\|_{L^1(-\ve(\mu),0)}\cdot\| M
 \left(\cdot, x, C_2 \xi(\cdot,x) \right)\|_{L^1\left({I}_i^\frac{1}{\mu}\cap [0,T]\right)}\,dx\\
 &\leq  C_1 \| M
 \left(\cdot, \cdot, C_2 \, \xi  \right)\|_{L^1(\OT)},
\end{split}
\]
what was to be proven.
\end{proof}

\begin{proof}[Proof of Theorem~\ref{theo:approx-t}] We extend $\vp\in\VTMi$ by $\vp(0,x)=\vp_0(x)$ on $(-\infty,0)$ and by $0$ on $(T,\infty)$. We shall prove that the sequences we look for are $\{ \vp_{\tm}\}_\tm$  coming from~\eqref{ximu}  and $\{ \vp^\bullet_{\tm}\}_\tm$  coming from~\eqref{xibmu}. Easy computation shows {\it i), ii), iii)}. 

 We concentrate now on showing {\it iv)}, i.e. the modular convergence
\[ \nabla(\vp^\bullet_{\tm})\xrightarrow[{\tm}\to\infty]{M}\nabla\vp\quad\text{ in }\quad L_M(\Omega_T;\rn),\]
 which suffices for $\lim_{\tm\to\infty}\vp^\bullet_\tm=  {\vp}$ strongly in $L^{1}(\Omega_T)$ due to Poincar\'e's inequality in $W^{1,1}_0(\Omega)$.

Let us consider a family of measurable sets  $\{ {E}_n \}_n$  such that $\bigcup_n {E}_n = \OT$ and a  vector valued simple function $ {E}^n(t,x)=\sum_{j=0}^n \mathds{1}_{{E}_j}(t,x) \va_{j}(t,x),$ converging modularly to $\nabla \vp $ with ${\lambda}_4$ (cf.~Definition~\ref{def:convmod}) which exists due to~Lemma~\ref{lem:dens}.    Note that
\[  \nabla (\vp^\bullet_{{\tm} })- \nabla \vp= \left( \nabla (\vp^\bullet_{\tm})-({E}^n)^\bullet_{{\tm} }\right) +( ({E}^n)^\bullet_\mu - {E}^n )
	+ ({E}^n  - \nabla\vp).\]
Convexity of $M(t,x, \cdot)$ implies
	\begin{equation*}%\label{IE:aw14}
	\begin{split}
	 \iOT M \left(t, x, \frac{ \nabla (\vp^\bullet_{{\tm} })- \nabla \vp }{ \lambda }\right) \,dx dt& \leq	\frac{ {\lambda}_1}{ {\lambda}} \iOT  M\left(t, x,
	\frac{\nabla( \vp^\bullet_{{\tm} })-( {E}^n)^\bullet_{{\tm} }}{  {\lambda}_1} \right) \,dx dt\\
	&\quad +\frac{ {\lambda}_2}{ {\lambda}} \iOT M\left( t,x,
	\frac{( {E}^n)^\bullet_{\mu} - {E}^n }{ {\lambda}_2} \right) \,dx dt+ \frac{ {\lambda}_3}{ {\lambda}} \iOT  M\left( t,x,  \frac{  {E}^n - \nabla  \vp }{ {\lambda}_3} \right) \,dx dt\\
	&=L^{ n, {\tm}}_1+L^{ n, \tm }_2+L^{\tm }_3,
	\end{split}
	\end{equation*}
where $ {\lambda}= \sum_{i=1}^3 {\lambda}_i$, $ {\lambda}_i>0$.  We have $ {\lambda}_3$ fixed already. Let us take $ {\lambda}_1= {\lambda}_3/C_2$.

In order to pass to the limit with ${\tm}\to\infty$, we apply Lemma~\ref{lem:step2prevt}  estimating \[0\leq \lim_{{\tm}\to\infty}L^{ n,{\tm}}_1\leq C_1 L^{ n }_3.\]  
 Furthermore,  Lemma~\ref{lem:dens} implies that $\lim_{n\to\infty}  L^{ n }_3= 0,$  which entails $\lim_{n\to\infty}  \limsup_{{\tm}\to \infty} L^{ n,{\tm}}_1= 0$ as well.

 Let us concentrate on $L^{ n,{\tm}}_2$.
The Jensen  inequality and then the Fubini theorem lead to
	\begin{equation}\label{IE:aw17'}
	\begin{split}
	&\frac{ {\lambda} }{ {\lambda}_2} L^{ n,{\tm}}_2  \\
	& =C\sum_{i=1}^{N_\tm^T}\int_{\Omega }\int_{I_i^\frac{1}{\tm}}
  M \left(t, x, \frac{1}{ {\lambda}_2} \int_{\r} {\tm} e^{{\tm}(s )}\mathds{1}_{(-\infty,0]}(s ) \sum_{j=0}^n [  \mathds{1}_{ {E}_j}(t,x)  \va_j (t,x)- \mathds{1}_{ {E}_j}(s-t,x)  \va_j (s-t,x) ]\,ds \right)\, dt\,dx\\ &
	\leq C
	  \sum_{i=1}^{N_\tm^T}\int_{\Omega }\int_{I_i^\frac{1}{\tm}} {\tm} e^{{\tm}(s )}\mathds{1}_{(-\infty,0]}(s )
	   M \left(t, x, \frac{1}{ {\lambda}_2} \sum_{j=0}^n [   \mathds{1}_{ {E}_j}(t,x)  \va_j (t,x) -\mathds{1}_{ {E}_j}(s-t,x)  \va_j(s-t,x) ]
	   \right) \,ds\,dt\,dx\\ &
	\leq C
	  \int_{\Omega_i}  \sum_{i=1}^{N_\tm^T}\int_{I_i^\frac{1}{\tm}}
	   M \left(t, x, \frac{1}{ {\lambda}_2} \sum_{j=0}^n [   \mathds{1}_{ {E}_j}(t,x)  \va_j (t,x) -\mathds{1}_{ {E}_j}(s-t,x)  \va_j(s-t,x) ]
	   \right) \,dt\,dx.
	\end{split}
	\end{equation}
We let ${\tm}\to\infty$. Notice that using the continuity of the shift operator in $L^1$ we observe that poinwisely
	\[  \sum_{j=0}^n [  \mathds{1}_{ {E}_j}(t,x)  \va_j (t,x)- \mathds{1}_{ {E}_j}(s-t,x)  \va_j(s-t,x)   ]\xrightarrow[{\tm}\to \infty]{} 0,\]
because $s-t<1/{\tm}$. Moreover, when we fix arbitrary $ {\lambda}_2>0$ we have
	\[\begin{split}
	 &M \left(t, x, \frac{1}{ {\lambda}_2} \sum_{j=0}^n [  \mathds{1}_{ {E}_j}(t,x)  \va_j (t,x)- \mathds{1}_{ {E}_j}(s-t,x)  \va_j(s-t,x)   ] \right) \leq \sup_{\eta\in\rn:\ | \eta|=1}M \left(t, x, \frac{1}{  {\lambda}_2} \sum_{j=0}^n|  \va_j|  \eta \right)  <\infty
	 \end{split}\]
and
the Lebesgue Dominated Convergence Theorem provides the right-hand side of \eqref{IE:aw17'} converges to zero.

Passing to the limit completes the proof of modular convergence of the approximating sequence.  

\medskip

Now we are going to show {\it v)}. The $L^\infty$ norm is preserved directly from the formula. Let us  note that
\[ \vp_\mu(t,x)-\vp^\bullet_\mu(t,x)=\mu\int_{-\infty}^{t-\ve(\mu)} e^{\mu(s-t)}\vp(s,x)ds
\]
and consequently, since here assume $\vp\in L^\infty(\OT)$, we have \[ \|\vp_\mu-\vp_\mu^\bullet\|_{L^\infty(\Omega_T)}\leq 
\|\vp \|_{L^\infty(\OT)}e^{-\mu\ve(\mu)}=
\|\vp \|_{L^\infty(\OT)}e^{-\log^2\mu}\xrightarrow[\mu\to\infty]{}0.\]
To justify that for every $\mu$ and a.e. $x\in\Omega$ function $\vp_\mu^\bullet(\cdot,x)\in W^{1,\infty}([0,T))$ we use Young's inequality for convolution of measures. Indeed, for every $\mu$ function $\pa_t\vp_\mu^\bullet(\cdot,x)$ has bounded total variation, because its accummulation points have finite mass.

Moreover, direct computation shows that
\[ \|\pa_t\big(\vp_\mu(\cdot,x)-\vp_\mu^\bullet(\cdot,x)\big)\|_{L^\infty(0,T)} \leq  
2\|\vp \|_{L^\infty(\OT)} \mu e^{-\mu\ve(\mu)}=
2\|\vp \|_{L^\infty(\OT)}e^{\log\mu(1-\log\mu)}\xrightarrow[\mu\to\infty]{}0\]
 uniformly for $x\in\Omega$, which completes the proof.
\end{proof}

%%%%%%%%%%%%%%%%%%%%%%%%%%%%%%%%%%%%%%%%%%%%%%%%%%%%%%%%
\section{The proof of main results -- existence of renormalized solutions} \label{sec:main proof}
%%%%%%%%%%%%%%%%%%%%%%%%%%%%%%%%%%%%%%%%%%%%%%%%%%%%%%%%
 In this section we provide proofs of the main goals -- existence of renormalized solutions Theorem~\ref{theo:aniso} and its special (isotropic) case~\ref{theo:main0}. It is obtained by steps. Recall that existence to bounded data problem is already given in Theorem~\ref{theo:bound}. We start with a priori estimates for truncations of solutions to the sequence of solutions to bounded-data problems. The use of truncation method at this stage is already classical and dates back to pioneering papers of~\cite{bgSOLA,boc-g-d-m,bbggpv}.

\subsection{Convergence of truncations $T_k (u_n)$}
\begin{prop}\label{prop:convTk} Suppose $[0,T]$ is a finite interval, $\Omega$ is a bounded Lipschitz domain in $ \rn$,   $N>1$,  $f\in L^1(\OT)$, $u_0\in L^1(\Omega)$, function $A$ satisfy assumptions ($\mathcal{A}$\ref{A1})--($\mathcal{A}$\ref{A3}) with an $N$-function~$M$ satisfying ($\mathcal{M}$) or ($\mathcal{M}_p$). Let  $u_n \in\VTMi$ denote a weak solution to the problem~\eqref{eq:bound}. Let $k>0$ be arbitrary. Then there exists a measurable function $u$, such that $T_k(u)\in\VTM$  and 
\begin{eqnarray} T_k(u_n) &\xrightharpoonup{} &  T_k(u)\quad \text{ in } L^1(0,T;W^{1,1}_0(\Omega)),\label{conv:n:TkuL1}\\
T_k(u_n) &\xrightharpoonup{*} &  T_k(u)\quad \text{weakly-* in } L^\infty (\OT),\label{conv:n:TkuLi}\\
\nabla T_k(u_n) &\xrightharpoonup* &\nabla T_k(u)\quad \text{weakly-* in } L_M(\OT;\rn),\label{conv:n:TkutM}\\
%\nabla T_k(u_n) &\xrightharpoonup{} &\nabla T_k(u)\quad \text{weakly in } L^1(\OT; \rn),\label{conv:n:TkutL1}\\
 A(t,x,\nabla T_k(u_n)) &\xrightharpoonup*& \Ak \quad \text{weakly-* in } L_{M^*}(\OT;\rn),\label{conv:n:ATkutMs}
\end{eqnarray}
for some $\Ak\in L_{M^*}(\OT;\rn)$.
\end{prop}
\begin{proof} Our aim is to apply integration-by-parts formula from Theorem~\ref{theo:intbyparts}   to $u_n$ being a weak solution to~\eqref{eq:bound} with a~special choice of the functions involved therein.  We already know that $T_k(u_n)\in V_T^M$ and $u_n(t,x)\in L^\infty([0,T];L^1(\Omega))$. Let two-parameter family of functions $\vt:\R\to\R$ be defined by
\begin{equation}
\label{vt}\vt(t):=\left(\omega_r * \mathds{1}_{[0,\tau)}\right)(t),
\end{equation}where $\omega_r$ is a standard regularizing kernel, that is $\omega_r\in C_c^\infty(\R)$, $\supp\,\omega_r\subset (-r,r)$. Note that $\supp\,\vt=[-r,\tau+r).$ In particular, for every $\tau$ there exists $r_\tau$, such that for all $r<r_\tau$ we have $\vt\in C_c^\infty([0,T))$.

We use  Theorem~\ref{theo:intbyparts}  with $A=A(t,x,\nabla u_n)$, $F=T_n(f)$, $h(\cdot)=T_k(\cdot)$, and $\xi(t,x)=\vt(t),$  we obtain
\[\begin{split}
 -\int_{\OT} \left(\int_{u_{0,n}(x)}^{u_n(t,x)} T_k(\s)d\s\right) \partial_t (\vt)   \,dx\,dt+\int_{\OT}A(t,x,\nabla u_n)\cdot \nabla (T_k(u_n)\vt) \,dx\,dt =\int_{\OT}T_n(f) T_k(u_n)\vt \,dx\,dt.\end{split}
\]

When we pass to the limit with ${r}\to 0$, for a.e. $\tau\in [0,T]$ we get
\[\begin{split}
 \int_{\Omega } \left(\int_{0}^{u_n(\tau,x)} T_k(\s)d\s-\int_{0}^{u_{0,n}(x)} T_k(\s)d\s\right)  \,dx +\int_{\Omega_\tau}A(t,x,\nabla u_n) \cdot \nabla  T_k(u_n) \,dx\,dt
 = \int_{\Omega_\tau}T_n(f) T_k(u_n)  \,dx\,dt,\end{split}
\]
and consequently
\[\begin{split}
 \frac{1}{2}\|T_k \left( u_n(\tau)\right)\|^2_{L^2(\Omega)} -\frac{1}{2}\|T_k \left( u_{0,n}\right)\|^2_{L^2(\Omega)}+\int_{\Omega_\tau}A(t,x,\nabla T_k (u_n)) \cdot \nabla T_k(u_n) \,dx\,dt= \int_{\Omega_\tau}T_n(f) T_k(u_n)\,dx\,dt.\end{split}\]
Recall that (${\cal A}$2) results in
\[c_AM^*(t,x,\nabla T_k (u_n))\leq M(t,x,\nabla T_k (u_n))\leq A(t,x,\nabla T_k (u_n))\nabla T_k (u_n).\] 
 Therefore we get
\[\begin{split}
&\frac{1}{2}\|T_k \left( u_n (\tau)\right)\|^2_{L^2(\Omega)} -\frac{1}{2}\|T_k \left( u_{0,n}\right)\|^2_{L^2(\Omega)}+\int_{\Omega_\tau}  M(t,x,\nabla  T_k (u_n)) \,dx\,dt \leq k\|f\|_{L^1(\OT)}  .\end{split}\]
When we notice that $\|T_k \left( u_{0,n}\right)\|^2_{L^2(\Omega)}\leq k\|  u_{0,n} \|_{L^1(\Omega)} = k\|  T_n(u_{0}) \|_{L^1(\Omega)}$, since $\tau\in( 0,T)$ is arbitrary, fixing \[w_2(k):= k\left(\|f\|_{L^1(\OT)}+\frac{1}{2}\|  u_{ 0} \|_{L^1(\Omega)}\right),\]  we infer
\begin{equation}\begin{split}
\int_{\Omega_\tau}  M(t,x,\nabla  T_k (u_n)) \,dx\,dt&\leq w_2(k),\\
c_A\int_{\Omega_\tau}  M^*(t,x, A (t,x, \nabla T_k (u_n)))\,dx\,dt&\leq w_2(k).\label{apriori}\end{split} 
\end{equation}
The weak lower semi-continuity of~a~convex functional together with the above a priori estimates imply existence of $u\in \VTM$ such that \eqref{conv:n:TkuL1},\eqref{conv:n:TkuLi},\eqref{conv:n:TkutM} hold, and existence of $\Ak$ such  that~\eqref{conv:n:ATkutMs} holds.\end{proof}

\subsection{Decay condition }

 \begin{prop}\label{prop:contr:rad:n}
Suppose $[0,T]$ is a finite interval, $\Omega$ is a bounded Lipschitz domain in $ \rn$,   $N>1$,  $f\in L^1(\OT)$, $u_0\in L^1(\Omega)$, function $A$ satisfy assumptions ($\mathcal{A}$\ref{A1})--($\mathcal{A}$\ref{A3}) with an $N$-function~$M$ satisfying ($\mathcal{M}$) or ($\mathcal{M}_p$). Assume further that $u_n$ is a weak solution to~\eqref{eq:bound}, $n>0$. Then
\begin{equation}
\lim_{l\to\infty} |\{|u_n|>l\}|=0.\label{conv:umeas:n}
\end{equation}
and \begin{equation}
 \label{eq:contr:rad:n}
 \lim_{l\to\infty}\limsup_{n\to\infty}\int_{\{l<|u_n|<l+1\}} A(t,x,\nabla u_n)\nabla u_n\,dx\,dt=0.
 \end{equation}
\end{prop}

\begin{proof} To prove~\eqref{conv:umeas:n} we consider   $\dm(s):= \inf_{(t,x)\in\OT,\,\xi:|\xi|=s} M(t,x,\xi)$ and $c^1_P,c^2_P>0$ being constants from Poincar\'e inequality (Theorem~\ref{theo:Poincare})  for which we have%\footnote{Jak by dac tu twierdzenie Cianchiego $W^{1}L_M\subset L_{M_n}$, to dostaniemy regularnosc w prz. Marcinkiewicza ${\cal M}^{M_n(\cdot)/\cdot}$; pytanie: czy rozwiazanie zagadnienia z ogr. danymi przez comparison priciple nie jest z urzedu ograniczone...} 
\[|\{|u_n|\geq l\}|=|\{|T_l(u_n)|= l\}|=|\{|T_l(u_n)|\geq l\}|=|\{\dm(c^1_P |T_l(u_n)|)\geq \dm(c^1_P l)\}|.\]
Moreover, for $l>1$ we have
\begin{equation*}
\begin{split}
|\{|u_n|\geq l\}|&\leq \int_{\OT} \frac{\dm(c^1_P|T_l(u_n)|)}{\dm(c^1_P l)}\, dx\,dt\leq \frac{c(N,\Omega,T)}{\dm(c^1_P l)} \int_\OT c^2_P \dm( |\nabla T_l(u_n)|)\,dx\,dt\\
&\leq    \frac{c(N,\Omega,T)}{\dm(c^1_P l)} \int_\OT M(t,x,  \nabla T_l(u_n) )dx\,dt  \\&
\leq \frac{C(M,N,\Omega,T)}{\dm(c^1_P l)}  \cdot l \left(\|f\|_{L^1(\OT)}+\frac{1}{2}\|  u_{ 0} \|^2_{L^2(\Omega)}\right)\\
& \leq C(f,u_0,M,N,\Omega,T)\frac{l}{\dm(c^1_P l)}  \xrightarrow[ l\to \infty]{}0
.\end{split}\end{equation*}
In the above estimates we applied the Chebyshev inequality, the Poincar\'{e} inequality (Theorem~\ref{theo:Poincare}), a~priori estimate~\eqref{apriori}, and the facts that $f\in  L^1(\OT),$ $u_0\in L^2(\Omega)$ and that $\dm$ is an $N$-function.

To prove~\eqref{eq:contr:rad:n}, we consider a family of nonincreasing functions $\phi_{r}\in C_c^\infty([0,T))$, such that
\[\phi_{r}(t):=\left\{\begin{array}{ll}
0& \text{for }t\in[T-{r},T],\\
1& \text{for }t\in[0,T-2{r}],
\end{array}\right.\qquad \text{and}\qquad G_l(s):=T_{l+1}(s)-T_{l}(s).\]

Since $u_n\in \VTMi$ is a weak solution to~\eqref{eq:bound}, we can use $\vp(t,x)=G_l(u_n(t,x))\phi_{r}(t)$ as a test function to get
\begin{equation*}
%\label{eq:bound:Glphi}
\iOT (\partial_t u_n) G_l(u_n) \phi_{r} \, dx\,dt+\int_{\{l<|u_n|<l+1\}} A(t,x,\nabla u_n)\nabla u_n\,\phi_{r}\, dx\,dt=\iOT fG_l(u_n) \phi_{r} \, dx\,dt.
\end{equation*}
Notice that on the left-hand side above we have
\[\begin{split}
\int_0^T (\partial_t u_n) G_l(u_n) \phi_{r} \,dt&=
\int_0^T \partial_t \left(\int_0^{u_n} G_l(s)ds\right)\, \phi_{r} \,dt =
-\phi_{r}(0)\int_0^{u_{0,n}}  G_l(s)ds -
\int_0^T  \int_0^{u_n} G_l(s)ds \,\partial_t \phi_{r} \,dt.\end{split}\]
Moreover, $\int_0^{u_n} G_l(s)ds\geq 0$ and $\partial_t \phi_{r}\leq 0$, hence \[-\int_0^T  \int_0^{u_n} G_l(s)ds \,\partial_t \phi_{r} \,dt\geq 0\]
and consequently
\[ \int_{\{l<|u_n|<l+1\}} A(t,x,\nabla u_n)\nabla u_n\,\phi_{r}\, dx\,dt\leq \iOT fG_l(u_n) \phi_{r} \, dx\,dt+\iO \phi_{r}(0)\int_0^{u_{0,n}}  G_l(s)ds\,dx.\]

Furthermore, to infer that the right-hand side above tends to zero when $l\to\infty$, it suffices to~observe that
\[\iO\int_0^{|u_{0,n}|} |G_l(s)|\,ds\,dx\leq \iO\int_0^{|u_{0,n}|} \mathds{1}_{\{s>l\}}\,ds\,dx=\int_{\{||u_{0,n}|-l|>0\}} \left(|u_{0,n}|-l\right) dx \xrightarrow[l\to\infty]{}0\]
and
\[\iOT fG_l(u_n) \phi_{r} \, dx\,dt\leq \int_{\{|u_n|>l\}} |f|\, dx\,dt\xrightarrow[l\to\infty]{}0.\]
Therefore,~\eqref{eq:contr:rad:n} follows.\end{proof}

\subsection{Almost everywhere limit}

 \begin{prop}\label{prop:contr:rad}
Suppose $[0,T]$ is a finite interval, $\Omega$ is a bounded Lipschitz domain in $ \rn$,   $N>1$,  $f\in L^1(\OT)$, $u_0\in L^1(\Omega)$, function $A$ satisfy assumptions ($\mathcal{A}$\ref{A1})--($\mathcal{A}$\ref{A3}) with an $N$-function~$M$ satisfying ($\mathcal{M}$) or ($\mathcal{M}_p$). Assume further that $u_n$ is a weak solution to~\eqref{eq:bound}.  For the  function $u \in \VTM $  coming from Proposition~\ref{prop:convTk} we have
\begin{equation}
 u_n\to u \quad  a.e.\ \text{in}\ \OT,\label{conv:usae}\end{equation}
and
\begin{equation}
\lim_{l\to\infty} |\{|u|>l\}|=0.\label{conv:umeas}
\end{equation}
\end{prop}

\begin{proof}  To prove~\eqref{conv:umeas} we apply the comparison principle (Proposition~\ref{prop:comp-princ}). We can do it since by Theorem~\ref{theo:intbyparts} weak solutions $u_n$ are renormalized ones. We define asymmetric truncations as follows\[T^{k,l}(f)(x)=\left\{\begin{array}{ll}
-k& f\leq -k,\\
f & |f|\leq k,\\
l&  f \geq l.
\end{array}\right.\]
 Let $u^{a,b}$ denote  a weak solution to \[u_t-{\dv }A(t,x,\nabla u )= T^{a,b}(f),\qquad u(0,x)=T^{a,b}(u_0),\] which exists according to Theorem~\ref{theo:bound}. For $0<l<l'$ and $0<k<k'$, comparison principle (Proposition~\ref{prop:comp-princ}) implies that
\begin{equation}
\label{order}
u^{k',l}\leq u^{k,l}\leq u^{k,l'}
\end{equation}
for a.e. $(t,x)\in\OT$. Due to the monotonicity of $(u^{k,l})_l$ we deduce that $\lim_{l\to\infty} u^{k,l}$ exists a.e. in~$\OT$. Let us denote it by $u^{k,\infty}$. On the other hand, taking into account~\eqref{order} we infer that $u^{k',\infty}\leq u^{k,\infty}$ a.e. in~$\OT$. Thus,  there exists the limit $u^{\infty,\infty}=\lim_{k\to\infty} u^{k,\infty}$ a.e. in~$\OT$. Consequently, due to~the~uniqueness of the limit (cf.~\eqref{conv:n:TkuL1}), we get the convergence~\eqref{conv:usae}. When we have~\eqref{conv:usae}, condition~\eqref{conv:umeas} is a direct consequence of~\eqref{conv:umeas:n}.\end{proof}

\subsection{Identification of the limit of $A(t,x,\nabla T_k(u_n))$}

In this step we employ  monotonicity trick to identify the limit~\eqref{conv:n:ATkutMs}. Let us stress that this is the part that is essentially more complex here, than in the case of the space not changing with time, e.g.~\cite{pgisazg2,gwiazda-ren-para}. Indeed, the classical tool of Landes regularization cannot be applied anymore and we need to use a much more subtle one coming from Theorem~\ref{theo:approx-t}.

\begin{prop} \label{prop:convTkII} Suppose $[0,T]$ is a finite interval, $\Omega$ is a bounded Lipschitz domain in $ \rn$,   $N>1$,  $f\in L^1(\OT)$, $u_0\in L^1(\Omega)$, function $A$ satisfy assumptions ($\mathcal{A}$\ref{A1})--($\mathcal{A}$\ref{A3}) with an $N$-function~$M$ satisfying ($\mathcal{M}$) or ($\mathcal{M}_p$). Suppose $u_n$ is a weak solution to~\eqref{eq:bound} and $k>0$ is arbitrary.  We have\begin{equation}
A(t,x,\nabla T_k(u_n)) \xrightharpoonup* A(t,x,\nabla T_k(u))\quad \text{weakly-* in } L_{M^*}(\OT;\rn).\label{conv:n:AT:id}
\end{equation}
\end{prop}
\begin{proof}
 We shall show, still for fixed $k$, that  in~\eqref{conv:n:ATkutMs} \begin{equation}
\label{Ak=ATk}\Ak=A(t,x,\nabla T_k(u)).
\end{equation}

Fix arbitrary nonnegative ${w}\in  C_c^\infty([0,T)).$ We show now that \begin{equation}
\label{limsup<} 
\limsup_{n\to\infty} \iOT {w} A(t,x,\nabla T_k(u_n))\cdot\nabla(T_k(u_n))\,dx\,dt\leq \iOT {w} \Ak \cdot\nabla(T_k(u))\,dx\,dt
\end{equation}
and then conclude~\eqref{Ak=ATk} via the monotonicity argument.

We consider the approximate sequence $\{\gkbm\}_\tm$ from Theorem~\ref{theo:approx-t}, such that $(\nabla T_k(u))^\bullet _\tm\xrightarrow{M} \nabla T_k(u)$ converges modularly in $L_M(\OT;\rn)$ when $\mu\to\infty$. In order to use Theorem~\ref{theo:intbyparts} to~\eqref{eq:bound},  we let $\psi_l:\R\to\R$ be given by
\begin{equation}\label{psil}
\psi_l(s):=\min\{(l+1-|s|)^+,1\}
\end{equation} and choose $A=A(t,x,\nabla u_n)\in L_M(\OT;\rn)$, $F=T_n f\in L^1(\OT) $. We apply it twice: first time with $h(\cdot)=\psi_l(\cdot) T_k(\cdot) $ and $\xi={w}$ and second time with  $h(\cdot)=\psi_l(\cdot)$ and $\xi={w}\gkbm.$ Subtracting the second from the first we get\begin{equation}
\label{1stsum:id:Ak}I_1^{n,\mu,l}+I_2^{n,\mu,l}+I_3^{n,\mu,l}=I_4^{n,\mu,l},
\end{equation}
where
\begin{eqnarray*}
I_1^{n,\mu,l}&=&-\iOT\partial_t{w}\int_{u_{0,n}}^{u_n}\psi_l(s) T_k(s)  ds\,dx\,dt
+\iOT\partial_t({w} \gkbm)\int_{u_{0,n}}^{u_n}\psi_l(s)ds\,dx\,dt,\\
I_2^{n,\mu,l}&=& \iOT{w} \psi_l(u_n)   A(t,x,\nabla  u_n)\cdot \nabla( T_k(u_n) -\gkbm) \,dx\,dt,\\
I_3^{n,\mu,l}&=& \iOT{w} \psi_l'(u_n) ( T_k(u_n) -\gkbm) A(t,x,\nabla  u_n)\cdot \nabla u_n \,dx\,dt,\\
I_4^{n,\mu,l}&= &\iOT{w} T_nf \psi_l(u_n) ( T_k(u_n) -\gkbm)  \,dx\,dt.
\end{eqnarray*}
We are going to pass to the limit with $n\to\infty$, then $\mu\to\infty$, and finally with $l\to\infty$. Roughly speaking we show that the limit of $I_1^{n,\mu,l}$ is nonnegative, then let $I_3^{n,\mu,l}$ and $I_4^{n,\mu,l}$ to zero. Then the limit of $I_2^{n,\mu,l}$ is nonpositive.

\medskip

\textbf{Limit of $I_1^{n,\mu,l}$. } We are going to prove that\begin{equation}\label{limI1>0}
\limsup_{l\to\infty}\limsup_{\mu\to\infty}\limsup_{n\to\infty} I_1^{n,\mu,l}\geq 0.
\end{equation} 
Let us consider a decomposition\begin{equation*}
 I_1^{n,\mu,l}=I_{1,1}^{n,\mu,l}+I_{1,2}^{n,\mu,l}+I_{1,3}^{n,\mu,l},
\end{equation*}
where, due to~\eqref{vp:ode}, we have
\begin{equation*}
\begin{split}
I_{1,1}^{n,\mu,l}&=-\iOT\partial_t{w}\int_{u_{0,n}}^{u_n}\psi_l(s)T_k(s) \,ds\,dx\,dt,\\
I_{1,2}^{n,\mu,l}&=\iOT(\partial_t w)\,\gkbm \left(\int_{u_{0,n}}^{u_n}\psi_l(s)\,ds \right)dx\,dt,\\
I_{1,3}^{n,\mu,l}&=\iOT  w\,\partial_t \big( \gkbm\big) \left(\int_{u_{0,n}}^{u_n}\psi_l(s)ds \right)dx\,dt.
\end{split}
\end{equation*}
Since $s\mapsto \psi_l(s)T_k(s)$ has a compact support, the convergence $u_n\to u$ a.e. in $\OT$ and continuity of the integral justify passing to the limit with $n\to\infty$ in $I_{1,1}^{n,\mu,l}$ to get
\[\lim_{n\to\infty}I_{1,1}^{n,\mu,l}=-\iOT\partial_t{w}\int_{u_{0}}^{u}\psi_l(s)T_k(s) ds\,dx\,dt=I_{1,1}^{l}.\]
Since
\begin{equation}\label{czesci}
 \int_0^{w}\psi_l(s)  T_k(s)  ds =\int_0^w\int_0^{T_k(s)}\psi_l(s) \,d\sigma\,ds
=T_k(v) \int_0^w\psi_l(s)  \,ds-\int_0^{ T_k(w)}\int_0^\sigma \psi_l(s)\,ds\,d\sigma,
\end{equation}
we can write
\begin{equation*}%\label{I1:parts}
\begin{split}
I_{1,1}^{l}&=  -\iOT\partial_t w\left( T_k(u) \int_{0}^{u}\psi_l(s)ds-
\int_0^{ T_k(u) }\int_{0}^{\sigma}\psi_l(s)ds\,d\sigma\right)dx\,dt\\
& \quad + \iO w(0)\left(\int_{0}^{ T_k(u_{0}) }\int_0^\sigma\psi_l(s)ds\,d\sigma-
 { T_k(u_{0}) }\int_{0}^{u_{0}}\psi_l(s)ds \right)dx.
\end{split}
\end{equation*}

In the case of $I_{1,2}^{n,\mu,l}$, according to the pointwise convergence of the integrand when $n\to\infty$ and boudedness of all of~the~involved terms Lemma~\ref{lem:TM1} justifies passing to the limit. Due to Theorem~\ref{theo:approx-t} {\it iv)} and {\it v)} we pass with $\mu\to\infty$  to get
\[\lim_{\mu\to\infty}\lim_{n\to\infty}I_{1,2}^{n,\mu,l} =\iOT(\partial_t w)\,T_k(u ) \left(\int_{0}^{u}\psi_l(s)ds-\int_{0}^{u_0}\psi_l(s)ds \right)dx\,dt=I_{1,2}^l.\]

As for $I_{1,3}^{n,\mu,l}$ we let $n\to\infty$ similarly to the case of $I_{1,2}^{n,\mu,l}$ and obtain
\[\begin{split}\lim_{n\to\infty}I_{1,3}^{n,\mu,l} &=\iOT  w\,\partial_t \big((T_k(u))_\mu\big) \left(\int_{u_{0}}^{u}\psi_l(s)ds \right)dx\,dt-\iOT  w \, \Big(\partial_t \big((T_k(u))_\mu-(T_k(u))_\mu^\bullet\big)\Big) \left(\int_{u_{0}}^{u}\psi_l(s)ds \right)dx\,dt\\
&= I_{1,3,1}^{ \mu,l}+I_{1,3,2}^{\mu,l}+
I_{1,3,3}^{\mu,l}+I_{1,3,4}^{\mu,l}+I_{1,3,5}^{\mu,l}. \end{split}\]
with
\begin{equation*}%\label{I12*}
\begin{split}
I_{1,3,1}^{\mu,l}&=\iOT  w\,\partial_t \big( \gkm
\big)
 \int_{T_k(u)}^{u} \psi_l(s)ds \, dx\,dt,\\
I_{1,3,2}^{ \mu,l}&=\iOT  w\,\partial_t \big( \gkm
\big)
 \int_{(g^k )_{\mu}}^{T_k(u) }  \psi_l(s)ds \, dx\,dt,\\
I_{1,3,3}^{ \mu,l}&=\iOT  w\,\partial_t \big( \gkm
\big) \int_0^{(g^k )_{\mu}}  \psi_l(s)ds \, dx\,dt,\\
I_{1,3,4}^{\mu,l}&=-\iOT  w\,\partial_t \big( \gkm
\big)
\int_0^{u_0} \psi_l(s)ds\, dx\,dt,\\
I_{1,3,5}^{\mu,l}&=-\iOT  w \, \Big(\partial_t \big((T_k(u))_\mu-(T_k(u))_\mu^\bullet\big)\Big) \left(\int_{u_{0}}^{u}\psi_l(s)ds \right)dx\,dt,
\end{split}
\end{equation*}
where due to~Theorem~\ref{theo:approx-t} {\it v)} we immediately realize that $\lim_{\mu\to\infty}I_{1,3,5}^{\mu,l}=0$.

We observe that after passing with $\mu\to\infty$, the formula from~\eqref{czesci} implies
\begin{equation}
\label{sum=0}
I_{1,1}^{ l}+I_{1,2}^{ l}+I_{1,3,3}^{ l}+I_{1,3,4}^{ l}=0.
\end{equation}
Indeed, convergence of $I_{1,3,3}^{ \mu,l}$ can be justified by integration by parts and continuity of the integral as follows
\[
\lim_{\mu\to\infty} I_{1,3,3}^{ \mu,l} =-\iOT  (\partial_t w)\int_0^{T_k(u)}\int_0^\sigma   \psi_l(s)ds \, d\sigma\, dx\,dt-\iO w(0)\int_0^{T_k(u_0)}\int_0^\sigma   \psi_l(s)ds \, d\sigma\, dx\,dt=I_{1,3,3}^{ l} .\]
In the case of $
I_{1,3,4}^{\mu,l}$ we integrate by parts and apply Lemma~\ref{lem:TM1} to get
\[\begin{split} \lim_{\mu\to\infty}I_{1,3,4}^{\mu,l}&=\lim_{\mu\to\infty}\iOT (\partial_t w) \gkm
\int_0^{u_0} \psi_l(s)ds\, dx\,dt+\iO w(0) T_k(u_0) \int_0^{u_0} \psi_l(s)ds\, dx\,dt\\&
= \iOT (\partial_t w) T_k(u)
\int_0^{u_0} \psi_l(s)ds\, dx\,dt+\iO w(0) T_k(u_0) \int_0^{u_0} \psi_l(s)ds\, dx\,dt =
I_{1,3,4}^{l}.\end{split}\]
Then summing all the terms we get~\eqref{sum=0}.

\medskip

Therefore, to get~\eqref{limI1>0} it suffices to show
\begin{equation}
\label{almostlimI1>0}
\limsup_{l\to\infty}\limsup_{\mu\to\infty} \Big(I_{1,3,1}^{\mu,l}+ I_{1,3,2}^{\mu,l}\Big)\geq 0.
\end{equation}
 We will do it using~\eqref{vp:ode}. Let us notice first that
 \[\begin{split}
I_{1,3,1}^{\mu,l}&=\iOT  w\,\mu \big(T_k(u)-\gkm
\big)
 \int_{T_k(u)}^{u} \psi_l(s)ds \, dx\,dt\\&=\int_{\{u \leq -k\}}  w\,\mu \big(-k-\gkm
\big)
 \int_{-k}^{u} \psi_l(s)ds \, dx\,dt+\int_{\{ u \geq k\}}  w\,\mu \big(k-\gkm
\big)
 \int_{k}^{u} \psi_l(s)ds \, dx\,dt\geq 0,\end{split}\]
 where we on $\{|u|\geq k\}$ the most internal integral collapses and each of the remaining terms is nonnegative. Moreover,  again due to~\eqref{vp:ode}, we have
 \[\begin{split}I_{1,3,2}^{ \mu,l}&=\iOT  w\,\mu \big(T_k(u)-\gkm
\big)
 \int_{\gkm}^{T_k(u) }  \psi_l(s)ds \, dx\,dt\geq 0.\end{split},\]
 which is justified by monotonicity of truncation.
\medskip

Thus, we have~\eqref{almostlimI1>0} and consequently~\eqref{limI1>0}.

 \bigskip

\textbf{Limit of $I_3^{n,\mu,l}$.} Since (A2) forces nonnegativeness of $A(t,x,\nabla  u_n)\cdot \nabla u_n$,  the radiation control~\eqref{eq:contr:rad:n} is equivalent to
\[
 \lim_{l\to\infty}\limsup_{n\to\infty}\int_{\{l<|u_n|<l+1\}} |A(t,x,\nabla u_n)\nabla u_n|\,dx\,dt=0.
 \]

Then
\[\begin{split}
|I_3^{n,\mu,l}|&=\left| \iOT{w} \psi_l'(u_n) \Big(T_k(u_n)-\gkbm\Big) A(t,x,\nabla  u_n)\cdot \nabla u_n \,dx\,dt\right|\leq\\
&\leq 2k||{w}||_{L^\infty(\R)} \int_{\{l<|u_n|<l+1\}} \left| A(t,x,\nabla  u_n)\cdot \nabla u_n \right|\,dx\,dt,
\end{split}\]
which is independent of $\mu$, so it implies \[
\lim_{l\to\infty}\lim_{\mu\to\infty}\limsup_{n\to\infty} I_3^{n,\mu,l}=0.\]

\textbf{Limit of $I_4^{n,\mu,l}$.} To deal with the limit with $n\to\infty$ we apply the Lebesgue Dominated Convergence Theorem due to the continuity of the integrand and~\eqref{conv:usae}, i.e.  $u_n\to u$ a.e. in $\OT$. Moreover, we know that  $\gkbm\to T_k(u)$ a.e. in $\OT$ when $\mu\to \infty$. Boundedness in $L^1$ of the rest terms enables to apply the Lebesgue Dominated Convergence Theorem and pass to the limit with $\mu\to \infty$ to get \[\lim_{l\to\infty}\lim_{\mu\to\infty} \lim_{n\to\infty} I_4^{n,\mu,l}=0.\]

\medskip

\textbf{Conclusion via the monotonicity trick. }  Recall~\eqref{1stsum:id:Ak}. Passing there to the limit, since $I_3^{n,\mu,l}$ and $I_4^{n,\mu,l}$ tend to zero,  we get
\[\begin{split}0=&\limsup_{l\to\infty}\limsup_{\delta\to 0}\limsup_{\tm\to\infty}\limsup_{n\to\infty} I_1^{n,\mu,l}\\+&\limsup_{l\to\infty}\limsup_{\delta\to 0}\limsup_{\tm\to\infty}\limsup_{n\to\infty} \iOT{w} \psi_l(u_n)   A(t,x,\nabla T_k( u_n))\cdot \nabla(T_k(u_n)-\gkbm)) \,dx\,dt+0.\end{split}\]
When we take into account~\eqref{limI1>0}, then the above line becomes\begin{equation*}
\limsup_{l\to\infty}\limsup_{\delta\to 0}\limsup_{\tm\to\infty}\limsup_{n\to \infty}\iOT {w}\psi_l(u_n) A(t,x,\nabla T_k(u_n))\cdot \nabla(T_k(u_n)-\gkbm)\,dx\,dt\leq 0.
\end{equation*}
Note that due to (A2) we have $A(t,x,\nabla T_k(u_n))\cdot \nabla(T_k(u_n)))\geq 0$ and $A(t,x,0)=0$. Therefore, for sufficiently large $l,\mu,n$, since $w,\psi_l\geq 0$, and~\eqref{psil}, we have
\[\iOT w A(t,x,\nabla T_k(u_n))\cdot \nabla(T_k(u_n))\,dx\,dt\leq
\iOT w A(t,x,\nabla T_k(u_n))\cdot \nabla(\gkbm)\,dx\,dt.\]
On the right-hand side above we use~that $
\nabla \gkbm\in L_M(\OT;\rn)$  and~\eqref{conv:n:ATkutMs}, and then for sufficiently large $\mu$
\[\limsup_{n\to\infty}\iOT w A(t,x,\nabla T_k(u_n))\cdot \nabla(T_k(u_n))\,dx\,dt\leq
\iOT w \Ak\cdot \nabla(\gkbm)\,dx\,dt.\]
Recall that $\nabla (\gkbm)\xrightarrow{M}\nabla T_k (u)$ (with $\tm\to\infty$). Therefore, the sequence $\{M(t,x,\nabla(\gkbm)/\lambda)\}_{\tm}$ is uniformly bounded in $L^1(\OT;\rn)$ for some $\lambda$ and consequently, by~Lemma~\ref{lem:unif}  $\{\nabla(\gkbm)\}_{\tm}$ is uniformly integrable. Hence the Vitali Convergence Theorem (Theorem~\ref{theo:VitConv}) gives
\[\lim_{\mu\to 0}\int_\OT w\Ak \cdot \nabla (\gkbm) \,dx=\int_\OT w \Ak \cdot \nabla T_k(u) \,dx.\]
Consequently, we obtain~\eqref{limsup<}. Following the monotonicity argument, as in the proof of Theorem~\ref{theo:bound}, we prove~\eqref{Ak=ATk}.  Monotonicity assumption ($\mathcal{A}3$) of~$A$ implies
\[(A(t,x,\nabla (T_k(u_n))-A(t,x,\eta))\cdot(\nabla(T_k( u_n))-\eta)\geq 0\qquad\text{a.e.  in }\OT\ \text{ for any }\eta\in L^\infty(\OT;\rn)\subset E_M(\OT;\rn).\] Since $A(\cdot,\cdot,\eta)\in L_{M^*}(\OT,\rn)= (E_{M}(\OT,\rn))^*$, we pass to the limit with $n\to\infty$ and take into account~\eqref{limsup<} to conclude that
\begin{equation}
\label{mono:int2}
\iOT w (\Ak-A(t,x,\eta))\cdot(\nabla (T_k(u))-\eta)\,dx\,dt\geq 0.
\end{equation}
Then  Lemma~\ref{lem:mon} with ${\cal{A}}={\cal A}_k$ and $\xi=\nabla (T_k(u))$ gives~\eqref{Ak=ATk}, which completes the proof of Proposition~\ref{prop:convTkII}. \end{proof}

\subsection{Main proof}

This part follows the ideas of~\cite{pgisazg2,gwiazda-ren-para}. We need to apply here the integration-by-parts formula, so indeed approximation from \cite{para-t-weak} is used and, consequently, we require condition (${\cal M}$), resp. (${\cal M}_p$). It is necessary to present this part, although there are no new challenges here.

\begin{proof}[{Proof of Theorem~\ref{theo:aniso}}] We shall use all the propositions of this section to get the final claim that the limit function $u$ from the claim of Proposition~\ref{prop:convTk} is the unique renormalized solution we look for.

\medskip

\textbf{Condition ($\mathcal{R}1$)}.\\ Obviously, when $u_n$ solves~\eqref{eq:bound} its limit $u$ satisfies condition ($\mathcal{R}1$), due to Propositions~\ref{prop:convTk} and~\ref{prop:convTkII}.

\medskip

The remaining ($\mathcal{R}2$)-($\mathcal{R}3$) require more arguments.

\medskip

\textbf{Condition ($\mathcal{R}3$)}.\\ The aim now is to prove the key convergence for condition~($\mathcal{R}3$), namely\begin{equation}
\label{adtkn}A(t,x,\nabla T_k(u_n))\cdot \nabla T_k(u_n)\xrightharpoonup{} A(t,x,\nabla T_k(u))\cdot \nabla T_k(u) \quad \text{weakly in }L^1(\OT).
\end{equation}
The reasoning involves the Chacon Biting Lemma and the Young measure approach. First we observe that  the sequence $\{[A(t,x,\nabla T_k(u_n))-A(t,x,\nabla T_k(u ))]\cdot [ \nabla T_k(u_n)- \nabla T_k(u_n)]\}_n$ is uniformly bounded in $L^1(\OT)$  due to~\eqref{apriori} and the Fenchel-Young inequality.

The monotonicity of $A(t,x,\cdot)$, uniform boundedness of $\{[A(t,x,\nabla T_k(u_n))-A(t,x,\nabla T_k(u ))]\cdot [ \nabla T_k(u_n)- \nabla T_k(u_n)]\}_n$ in $L^1(\OT)$, and Theorem~\ref{theo:Youngmeas} combined with Theorem~\ref{theo:bitinglemma1} give, up to~a~subsequence, convergence
\begin{equation}
\label{110}
\begin{split}0&\leq {w} [A(t,x,\nabla T_k(u_n) )-A(t,x,\nabla T_k(u ))]\cdot [\nabla T_{k}(u_n)-\nabla T_{k}(u )]\\
&\xrightarrow{b} w \int_\rnt [A(t,x,\lambda)-A(t,x,\nabla T_{k}(u ))]\cdot [\lambda-\nabla T_{k}(u )]d\nu_{t,x}(\lambda),\end{split}
\end{equation}
where $\nu_{t,x}$ denotes the Young measure generated by the sequence $\{\nabla T_{k}(u_n)\}_n$.

Since $\nabla T_{k}(u_n)\xrightharpoonup{}\nabla T_{k}(u )$ in $L^1(\OT)$, we have $\int_\rnt\lambda \, d\nu_{t,x}(\lambda)=\nabla T_{k}(u)$ for a.e. $t\in(0,T)$ and a.e. $x\in\Omega$. Then
\[\int_\rnt  A(t,x,\nabla T_k(u) ) \cdot [\lambda-\nabla T_{k}(u )]d\nu_{t,x}(\lambda)=0\]
and the limit in~\eqref{110} is equal for a.e. $t\in(0,T)$ and a.e. $x\in\Omega$ to
\begin{equation}\begin{split}
\label{111} {w} \int_\rnt [A(t,x,\lambda)-A(t,x,\nabla T_{k}(u ))]\cdot [\lambda-\nabla T_{k}(u )]d\nu_{t,x}(\lambda)\\=
{w}\int_\rnt  A(t,x,\lambda) \cdot  \lambda\, d\nu_{t,x}(\lambda)-{w}\int_\rnt  A(t,x,\lambda) \cdot \nabla  T_{k}(u )d\nu_{t,x}(\lambda).\end{split}
\end{equation}

Uniform boundedness of the sequence $\{ A(t,x,\nabla T_k(u_n))\cdot \nabla T_{k}(u_n) \}_n$ in $L^1(\OT)$  enables us  to apply once again Theorem~\ref{theo:Youngmeas} combined with Theorem~\ref{theo:bitinglemma1} to obtain
\begin{equation*}
%\label{Yinq}
 A(t,x,\nabla T_k(u_n))\cdot \nabla T_{k}(u_n)\xrightarrow{b} \int_\rnt  A(t,x,\lambda) \cdot  \lambda\,d\nu_{t,x}(\lambda).
\end{equation*}
Moreover, assumption (A2) implies $A(t,x,\nabla T_k(u_n))\cdot \nabla T_{k}(u_n)\geq 0$. Therefore, due to~\eqref{111} and~\eqref{110}, we have
\[\limsup_{n\to\infty} A(t,x,\nabla T_{k}(u_n) ) \nabla T_{k}(u_n)\geq \int_\rnt  A(t,x,\lambda) \cdot  \lambda\,d\nu_{t,x}(\lambda).\]
Taking into account that in~\eqref{conv:n:AT:id} we can put ${\cal A}_k=A(t,x,\nabla T_{k}(u) )=\int_\rnt  A(t,x,\lambda) \,d\nu_{t,x}(\lambda)$, the above expression implies\[\nabla T_{k}(u) \int_\rnt  A(t,x,\lambda) \,d\nu_{t,x}(\lambda)\geq \int_\rnt  A(t,x,\lambda) \cdot  \lambda\,d\nu_{t,x}(\lambda).\]

When we apply it, together with~\eqref{111}, the limit in~\eqref{110} is non-positive. Hence,
\begin{equation*} [A(t,x,\nabla T_{k}(u_n) )-A(t,x,\nabla T_{k}(u ))]\cdot [\nabla T_{k}(u_n)-\nabla T_{k}(u )]\xrightarrow{b} 0.
\end{equation*}
Observe further that $A(t,x,\nabla T_{k}(u ))\in L_{M^*}(\OT;\rn)$ and we can choose ascending family of sets $E^{k}_j$, such that $| E^{k}_j|\to 0$ for $j\to \infty$ and $A(t,x,\nabla T_{k}(u ))\in L^\infty(\OT\setminus E^{k}_j).$ Then, since $\nabla T_{k}(u_n)\xrightharpoonup{}\nabla T_{k}(u )$, we get\begin{equation*} A(t,x,\nabla T_{k}(u )) \cdot [\nabla T_{k}(u_n)-\nabla T_{k}(u )]\xrightarrow{b}  0%\xrightharpoonup{}0\quad\text{weakly in } L^1(\Omega)
\end{equation*}
and similarly we conclude $ A(t,x,\nabla T_{k}(u_n) )\cdot\nabla T_{k}(u )\xrightarrow{b}   A(t,x,\nabla T_{k}(u))\cdot\nabla T_{k}(u ).$ Summing it up we get
\begin{equation*}%\label{116}
A(t,x,\nabla T_{k}(u_n ))\cdot\nabla T_{k}(u_n )\xrightarrow{b}  A(t,x,\nabla T_{k}(u))\cdot\nabla T_{k}(u ).
\end{equation*}
Let us point out that ($\mathcal{A}2$) ensures that both --- the right and the left--hand sides are nonnegative. Recall that Theorem~\ref{theo:bitinglemma} together with~\eqref{limsup<} and~\eqref{conv:n:ATkutMs} results in~\eqref{adtkn}.

\medskip

Note that $\nabla u_n=0$ a.e. in $\{|u_n|\in\{l,l+1\}\}$. Then~\eqref{eq:contr:rad:n} implies
\begin{equation*}
%\label{127}
 \lim_{l\to\infty}\sup_{n>0}\int_{\{l-1<|u_n|<l+2\}}A(t,x,\nabla u_n)\cdot\nabla u_n\,dx=0.
\end{equation*}
For $g^l:\r\to\r$ defined by
\[g^l(s)=\left\{\begin{array}{ll}1&\text{if }\ l\leq| s|\leq l+1,\\
0&\text{if }\ |s|<l-1\text{ or } |s|> l+2,\\
\text{is affine} &\text{otherwise},
\end{array}\right.\]
we have
\begin{equation}
\label{128}
\int_{\{l<|u|<l+1\}}A(t,x,\nabla u)\cdot\nabla u\,dx\,dt\leq \int_{\OT}g^l(u)A(t,x,\nabla T_{l+2}( u))\cdot\nabla  T_{l+2}( u)\,dx\,dt.
\end{equation}
Let us remind.~\eqref{conv:usae} gives $u_n\to u$ a.e. in $\OT$, while~\eqref{conv:umeas} provides $\lim_{l\to\infty}|\{x:|u_n|>l\}|=0$. Moreover, we have weak convergence~\eqref{adtkn}, $A(t,x,\nabla T_{l+2}( u_n))\cdot\nabla  T_{l+2}( u_n)\geq 0$ and function $g^l$ is continuous and bounded. Thus, we estimate the limit of the right-hand side of~\eqref{128} in the following way
\[\begin{split}
0&\leq \lim_{l\to\infty} \int_{\{l-1<|u|<l+2\}}A(t,x,\nabla u)\cdot\nabla u\,dx\,dt\leq \lim_{l\to\infty}\int_{\Omega}g^l(u)A(t,x,\nabla T_{l+2}( u))\cdot\nabla  T_{l+2}( u)\,dx\,dt=\\
&= \lim_{l\to\infty} \lim_{n\to\infty} \int_{\Omega}g^l(u_n)A(t,x,\nabla T_{l+2}(u_n))\cdot\nabla T_{l+2}(u_n)\,dx\,dt\leq\\
&\leq \lim_{l\to\infty}  \lim_{n\to\infty}\int_{\{l-1<|u_n|<l+2\}}A(t,x,\nabla T_{l+2}(u_n))\cdot\nabla T_{l+2}(u_n)\,dx\,dt=0,
\end{split}\]
where the last equality comes from~\eqref{eq:contr:rad:n}. Hence, our solution $u$ satisfies condition ($\mathcal{R}3$).

\medskip

\textbf{Condition ($\mathcal{R}2$). }\\
 We apply Theorem~\ref{theo:intbyparts} for~\eqref{eq:bound}, arbitrary $h\in C_c^1(\R)$ and $\xi\in{C_c^\infty}([0,T)\times\Omega)$ obtaining
\begin{equation}
\label{fin-weak-n}
-\int_{\OT} \left(\int_{u_{0,n}}^{u_n} h(\s)d\s\right) \partial_t \xi \ \,dx\,dt+\int_{\OT}A(t,x,\nabla u_n)\cdot \nabla (h(u_n)\xi) \,dx\,dt=\int_{\OT}T_n(f) h(u_n)\xi \,dx\,dt.
\end{equation}
To pass to the limit with $n\to\infty$  side above we fix $R>0$ such that $\supp\, h\subset [-R,R]$. The right-hand converges to the desired limit due to the Lebesgue Dominated Convergence Theorem since $T_n f\to f$ in $L^1(\OT)$ and $\{h(u_n)\}_n$ is uniformly bounded.

To pass to the limit on the left-hand side  we notice that  we have there
\[\lim_{n\to\infty}-\int_{\OT} \left(\int_{u_{0,n}}^{u_n} h(\s)d\s\right) \partial_t \xi \ \,dx\,dt=-\int_{\OT} \left(\int_{u_{0 }}^{u } h(\s)d\s\right) \partial_t \xi \ \,dx\,dt, \]
where the equality is justified by the continuity of the integral. As for the second expression, we can write
\[\begin{split}&\int_{\OT}A(t,x,\nabla u_n)\cdot \nabla (h(u_n)\xi) \,dx\,dt
= \int_{\OT}h'(T_R(u_n)) A(t,x,\nabla T_R(u_n))  \nabla T_R(u_n)\,\xi \,dx\,dt\\&\qquad+\int_{\OT}h(T_R(u_n))A(t,x,\nabla T_R(u_n))\cdot \nabla  \xi  \,dx\,dt=III^n_1+III^n_2.\end{split}\]
%where $\supp\xi\subset[0,\tau)\times\Omega$ for some $\tau\in(0,T)$.
 Recall weak convergence of $A(t,x,\nabla T_k(u_n))\cdot \nabla T_k(u_n)$ in $L^1(\OT)$~\eqref{adtkn}. Since $h'(u_n)\xi\to h'(u)\xi$ a.e. in~$\OT$ and \[\|h'(u_n)\xi\|_{L^\infty(\OT)}\leq \|h'(u_n) \|_{L^\infty(\OT)}\| \xi\|_{L^\infty(\OT)},\]
we pass to the limit with $n\to\infty$ in $III^n_1$. To complete the case of $III^n_2$ we observe that Proposition~\ref{prop:convTkII}  implies weak convergence of $A(t,x,\nabla T_R(u_n))$ in $L^1(\OT)$ as $n\to\infty$. Moreover, $\{h(T_R(u_n))\}_n$ converges a.e. in~$\OT$ to $h(T_R(u))$ and is uniformly bounded in $L^\infty(\OT)$, so we can pass to the limit. Altogether  we have
\[\lim_{n\to\infty}(III^n_1+III^n_2)=\int_{\OT}h'(T_R(u )) A(t,x,\nabla u )  \nabla T_R(u )\,\xi \,dx\,dt +\int_{\OT}h(T_R(u ))A(t,x,\nabla T_R(u ))\cdot \nabla  \xi  \,dx\,dt.\]
Therefore, all the expressions of~\eqref{fin-weak-n} converge to the limits as expected in ($\mathcal{R}$2).

We already proved that $u$ satisfies ($\mathcal{R}$1), ($\mathcal{R}$2), and ($\mathcal{R}$3), hence it is a~renormalized solution. Uniqueness is a direct consequence of the comparison principle (Proposition~\ref{prop:comp-princ}).  \end{proof}

\begin{proof}[Proof of Theorem~\ref{theo:main0}] In isotropic spaces, we repeat the observation of~\cite[Lemma~6.4]{para-t-weak} that the condition $({\cal M}^{iso})$ (resp. $({\cal M}^{iso}_p)$) results from $({\cal M})$ (resp. $({\cal M}_p)$) and therefore, Theorem~\ref{theo:main0} is a direct consequence of~Theorem~\ref{theo:aniso}.  
\end{proof}

%%%%%%%%%%%%%%%%%%%%%%%%%%%%%%%%%%%%%%%%%%%%%%%%%%%%%%%%%%%%%%%%%%%%%%%%%%%%%%%%%
\section{Appendix}\label{sec:appendix}
%%%%%%%%%%%%%%%%%%%%%%%%%%%%%%%%%%%%%%%%%%%%%%%%%%%%%%%%%%%%%%%%%%%%%%%%%%%%%%%%%

\begin{lem} \label{lem:TM1}
 Suppose $w_n\xrightharpoonup[n\to\infty]{}w$ in $L^1(\Omega_T)$, $v_n,v\in L^\infty(\Omega_T)$, and $v_n\xrightarrow[n\to\infty]{a.e.}v$. Then \[\int_\OT w_n v_n\,dx\,dt \xrightarrow[n\to\infty]{}\int_\OT w v\,dx\,dt.\]
\end{lem}

\begin{defi}[Uniform integrability] We call a sequence  $\{f_n\}_{n=1}^\infty$ of measurable functions $f_n:\OT\to \rn$
 uniformly integrable if
\[\lim_{R\to\infty}\left(\sup_{n\in\mathbb{N}}\int_{\{x:|f_n(x)|\geq R\}}|f_n(x)|\,dx\,dt\right)=0.\]
 \end{defi}

%We have two equivalent definitions of modular convergence.
\begin{defi}[Modular convergence]\label{def:convmod}
We say that a sequence $\{\xi_i\}_{i=1}^\infty$ converges modularly to $\xi$ in~$L_M(\OT;\rn)$ (and denote it by $\xi_i\xrightarrow[i\to\infty]{M}\xi$), if
\begin{itemize}
\item[i)] there exists $\lambda>0$ such that
\begin{equation*}
%\label{convmodi}
\int_{\OT}M\left(t,x,\frac{\xi_i-\xi}{\lambda}\right)\,dx\,dt\to 0,
\end{equation*}
equivalently
\item[ii)] there exists $\lambda>0$ such that
\begin{equation*}
%\label{convmodii}
 \left\{M\left(t,x,\frac{\xi_i}{\lambda}\right)\right\}_i \ \text{is uniformly integrable in } L^1(\OT)\quad \text{and}\quad \xi_i\xrightarrow[]{i\to\infty}\xi \ \text{in measure};
\end{equation*}
\end{itemize}
\end{defi}

%We use the following results.
\begin{lem}[Modular-uniform integrability,~\cite{gwiazda2}]\label{lem:unif}
Let $M$ be an $N$-function and $\{f_n\}_{n=1}^\infty$ be a~sequence of measurable functions such that $f_n:\OT\to \rn$ and $\sup_{n\in\N}\int_\OT M(t,x,f_n(x))\,dx\,dt<\infty$. Then the sequence $\{f_n\}_{n=1}^\infty$ is uniformly integrable.
\end{lem}

The following result can be obtained by the method of the proof of~\cite[Theorem~7.6]{Musielak}.
\begin{lem}[Density of simple functions, \cite{Musielak}]\label{lem:dens}
Suppose~\eqref{ass:M:int}. Then the set of simple functions integrable on $\OT$ is dense in $L_M(\OT)$ with respect to the modular topology.
\end{lem}

\begin{defi}[Biting convergence]\label{def:convbiting}
Let $f_n,f\in  L^1(\OT)$ for every $n\in\N$. We say that a sequence $\{f_n\}_{n=1}^\infty$ converges in the sense of biting to $f$ in~$L^1(\OT)$ (and denote it by $f_n\xrightarrow[]{b}f$), if  there exists a sequence of measurable $E_k$ -- subsets of $\OT$, such that $\lim_{k\to\infty} |E_k|=0$, such that for every $k$ we have $f_n\to f$ in $L^1(\OT\setminus E_k)$.
\end{defi}

To present basic information on the Young measures, let us denote the space of signed Radon measures with finite mass by ${\cal M}(\rn)$.
\begin{theo}[Fundamental theorem on the Young measures]\label{theo:Youngmeas}
Let $U\subset\rn$ and $z_j:U\to\rn$ be a~sequence of measurable functions. Then there exists a~subsequence $\{z_{j,k}\}$ and a~family of weakly-* measurable maps $\nu_x:u\to{\cal M}(\rn)$, such that:
\begin{itemize}
\item $\nu_x\geq 0$, $\|\nu_x\|_{{\cal M}(\rn)}=\int_\rn d\nu_x\leq 1$ for a.e. $x\in U$.
\item For every $f\in C_0(\rn)$, we have $f(z_{j,k})\xrightharpoonup[]{*} \bar{f}$ in $L^\infty(U)$. Moreover,  $\bar{f}(x)=\int_\rn f(\lambda) \, d\nu_x(\lambda)$.
\item Let $K\subset\rn$ be compact. Then $\supp\,\nu_x\subset K,$ if $dist(z_{j,k},K)\to 0$ in measure.
\item $\|\nu_x\|_{{\cal M}(\rn)}=1$ for a.e. $x\in U$ if and only if the tightness condition is satisfied, that is $\lim_{R\to\infty}\sup_k|\{|z_{j,k}|\geq R\}|=0$.
\item If the tightness condition is satisfied, $A\subset U$ is measurable, $f\in C(\rn)$, and $\{f(z_{j,k})\}$ is relatively weakly compact in $L^1 (A)$, then $f(z_{j,k})\xrightharpoonup[]{} \bar{f}$ in $L^1(A)$ and  $\bar{f}(x)=\int_\rn f(\lambda) \, d\nu_x(\lambda)$.
\end{itemize}
The family of maps $\nu_x:U\to {\cal M}(\rn)$ is called the
Young measure generated by the sequence $\{z_{j,k}\}$.
\end{theo}

\begin{theo}[The Chacon Biting Lemma, cf. Theorem~6.6 in \cite{pedr}]\label{theo:bitinglemma1}Let the sequence $\{f_n\}_n$ be uniformly bounded in $L^1(\OT)$. Then there exists $f\in L^1(\OT)$, such that $f_n\xrightarrow[]{b}f$.
\end{theo}

The consequence of the above result is the following, cf.~\cite[Lemma~6.9]{pedr}.

\begin{theo}\label{theo:bitinglemma}Let $f_n\in   L^1(\OT)$ for every $n\in\N$,   $f_n\geq 0$ for every $n\in\N$ and a.e. in $\OT$. Moreover, suppose $f_n\xrightarrow[]{b}f$ and $\limsup_{n\to\infty}\int_\OT f_n \,dx\,dt\leq \int_\Omega f \,dx\,dt.$ Then  $f_n\xrightharpoonup{}f$ in $L^1(\OT)$ for $n\to\infty$.
\end{theo}

\begin{theo}[The Vitali Convergence Theorem]\label{theo:VitConv} Let $(X,\mu)$ be a positive measure space, $\mu(X)<\infty $, and $1\leq p<\infty$. If $\{f_{n}\}$ is uniformly integrable in $L^p_\mu$,   $f_{n}(x)\to f(x)$ in measure  and $|f(x)|<\infty $  a.e. in $X$, then  $f\in  {L}^p_\mu(X)$
and  $f_{n}(x)\to f(x)$ in  ${L}^p_\mu(X)$.
\end{theo}

\section{References}
\bibliographystyle{plain}
\bibliography{arXiv-Para-t}
\end{document}